\documentclass[journal]{IEEEtran}
\usepackage{ifthen}
\newboolean{showcomments}
\setboolean{showcomments}{true}
\usepackage[usenames,dvipsnames]{color}

\definecolor{bleudefrance}{rgb}{0.19, 0.55, 0.91}
\definecolor{ao(english)}{rgb}{0.0, 0.5, 0.0}

\newcommand{\addcite}[0]{\ifthenelse{\boolean{showcomments}}
{\textcolor{purple}{(add cite(s)) }}{}}%

\newcommand{\enrique}[1]{  \ifthenelse{\boolean{showcomments}}
{{\color{blue} Enrique: #1}}{}}
\newcommand{\emmargin}[1]{\ifthenelse{\boolean{showcomments}}{\marginpar{\color{bleudefrance}\tiny EM: #1}}{}}
\newcommand{\dennice}[1]{\ifthenelse{\boolean{showcomments}}
{\textcolor{blue}{Dennice says: #1}}{}}

\newboolean{showedits}
\setboolean{showedits}{false}
\usepackage[markup=underlined]{changes}
\definechangesauthor[color=bleudefrance]{EM}
\newcommand{\aem}[1]{
\ifthenelse{\boolean{showedits}}
{\added[id=EM]{#1}}
{\!#1\hspace{-4.75pt}}
}
\newcommand{\repem}[2]{
\ifthenelse{\boolean{showedits}}
{\replaced[id=EM]{#1}{#2}}
{\!#1\hspace{-4.75pt}}
}
\newcommand{\dem}[1]{
\ifthenelse{\boolean{showedits}}
{\deleted[id=EM]{#1}}
{}
}
\newcommand{\hl}[1]{\ifthenelse{\boolean{showcomments}}
{\textcolor{purple}{#1}}{}}%

\usepackage{epsfig, latexsym, times, bbm}
\usepackage{color}
\usepackage{colortbl}
\usepackage{array}
\usepackage{cite}
\usepackage{amsmath}
\usepackage{graphics}
\usepackage{graphicx}
\usepackage{amssymb}
\usepackage{url}
\usepackage{comment}
\usepackage{multirow}

\usepackage{bbding}
\usepackage{amsfonts}
\usepackage{cases}
\usepackage{subfigure}
\usepackage{algorithmic}
\usepackage{mathrsfs}
\usepackage{bm}
\usepackage{verbatim}
\usepackage{enumerate}
\usepackage{booktabs}
\usepackage[linesnumbered,ruled,longend]{algorithm2e}
\usepackage{textcomp}
\usepackage{subeqnarray}
\usepackage{amsthm}
\def\ba{\begin{array}}
	\def\ea{\end{array}}
\newcommand{\beq}{\begin{equation}}
	\newcommand{\eeq}{\end{equation}}
\newcommand{\bq}{\begin{eqnarray}}
	\newcommand{\eq}{\end{eqnarray}}
\newcommand{\bqn}{\begin{eqnarray*}}
	\newcommand{\eqn}{\end{eqnarray*}}
\newcommand{\bee}{\begin{enumerate}}
	\newcommand{\eee}{\end{enumerate}}
\newcommand{\bi}{\begin{itemize}}
	\newcommand{\ei}{\end{itemize}}
\newcommand{\bseq}{\begin{subequations}}
\newcommand{\eseq}{\end{subequations}}

\newcommand{\diag}{\textrm{diag}}
\newcommand{\blockdiag}{\textrm{blockdiag}}
\newtheorem{definition}{\textbf{Definition}}
\newtheorem{proposition}{\textbf{Proposition}}
\newtheorem{corollary}{\textbf{Corollary}}
\newtheorem{lemma}{\textbf{Lemma}}
\newtheorem{theorem}{\textbf{Theorem}}
\newtheorem{remark}{\textbf{Remark}}
\newtheorem{assumption}{\textbf{Assumption}}
\usepackage{setspace}
\ifCLASSINFOpdf
\else
\fi
\begin{document}

\title{On the Stability, Economic Efficiency and Incentive Compatibility of Electricity Market Dynamics}

\author{Pengcheng~You,~\IEEEmembership{Member,~IEEE,}
        Yan~Jiang,~
        Enoch~Yeung,~\IEEEmembership{Member,~IEEE,}
        Dennice~F.~Gayme,~\IEEEmembership{Senior~Member,~IEEE,}
        and~Enrique~Mallada,~\IEEEmembership{Senior~Member,~IEEE}
\thanks{}
\thanks{}
\thanks{}
\thanks{P. You, Y. Jiang, D. F. Gayme and E. Mallada are with the Whiting School of Engineering, Johns Hopkins University, Baltimore, MD 21218, USA (email: \{pcyou, yjiang, dennice, mallada\}@jhu.edu).}
\thanks{E. Yeung is with the Department of Mechanical Engineering, UC Santa Barbara, Santa Barbara, CA 93106, USA (email: eyeung@engineering.ucsb.edu).}
\thanks{A preliminary version of this work was presented in ACM e-Energy 2018~\cite{you18stabilization}.}
}


\IEEEoverridecommandlockouts
\allowdisplaybreaks

\maketitle

\thispagestyle{plain} 
\pagestyle{plain}

\begin{abstract}
This paper focuses on the operation of an electricity market that accounts for participants that bid at a sub-minute timescale. To that end, we model the market-clearing process as a dynamical system, called \emph{market dynamics}, which is temporally coupled with the grid frequency dynamics and is thus required to guarantee system-wide stability while meeting the system operational constraints. We characterize participants as price-takers who rationally update their bids to maximize their utility in response to real-time schedules of prices and dispatch. For two common bidding mechanisms, based on quantity and price, we identify a notion of alignment between participants' behavior and planners' goals that leads to a saddle-based design of the market that guarantees convergence to a point meeting all operational constraints. We further explore cases where this alignment property does not hold and observe that misaligned participants' bidding can destabilize the closed-loop system.  We thus design a regularized version of the market dynamics that recovers all the desirable stability and steady-state performance guarantees. Numerical tests validate our results on the IEEE 39-bus system.

\end{abstract}


\IEEEpeerreviewmaketitle

\section{Introduction}


\IEEEPARstart{E}{lectricity} markets aim to foster competition by allowing participants to make individual bids in the market clearing process \cite{kirschen2018fundamentals}. 
The shift of the electricity generation mix towards intermittent less-controllable renewable sources requires electricity markets to exploit resources with fast-acting capabilities.
However, the traditional market clearing process for economic dispatch, that spans five-minute  intervals or longer, is unable to make full use of such fast resources.
This limitation poses a challenge on the efforts to incentivize their market participation.
Faster market clearing, e.g., at the sub-minute level, provides a potential solution by (a) increasing flexibility to better accommodate the increasing system variability; and (b) increasing efficiency with finer-granularity economic dispatch.
However, such a fast-timescale market clearing process could interfere with the grid reliability, e.g., when market actions interact with the electromechanical swings of generators~\cite{alvarado2001stability}, thus raising concerns of grid operators.
Fast-timescale markets therefore have to further account for stability issues while pursuing economic efficiency.
This interplay, between physics (grid dynamics) and economics (market operation), overlaps with the existing cross-timescale frequency control architecture of grid operators~\cite{machowski1997power,wood2013power}: primary -- frequency regulation (10s of seconds), secondary -- nominal frequency restoration (1 minute), and tertiary -- economic dispatch (5+ minutes).
While there have been many recent efforts to integrate these temporally decoupled architecture, exploiting hierarchical structures and optimization decomposition across space~\cite{zhao2016unified,li2016connecting,mallada2017optimal,wang2017distributed1,wang2017distributed2,weitenberg2018robust,pang2017optimal,Dorfler_2016} and time~\cite{cai2017distributed}, such approaches follow an engineering perspective and render control laws on participants that do not necessarily reflect their individual economic incentives.


Instead, in this work we take into account participants' incentives and seek to incorporate the cross-timescale goals of frequency control in the market clearing process.
In this setting, \emph{participants can bid in real time},
and the market undertakes the role of ensuring economic efficiency and meeting a wide set of operational constraints (frequency regulation, power flow bounds, etc.) via pricing and dispatching.
In particular, we aim to design a fast-timescale sub-minute market that uses market signals as control signals and thereby operates as a controller. Thus, the market rules can be seen as a dynamical system, which is usually referred to as \emph{market dynamics}~\cite{alvarado2001stability,jin2013designing,watts2004influence,mota2001dynamic,kiani2010effect,liang2015stability}. 
The fundamental challenge for such a market is to simultaneously account for the physical response of the power grid and the economic incentives of its participants, e.g., generators. 
We model 
market dynamics with bids, based on quantity or price, as inputs, reflecting participants' preferences, and prices and dispatch as outputs. 
The market is dynamically coupled with participants (that bid according to their own preferences) and the power grid (that is constrained by Newton's and Kirchhoff's Laws). 

The notion of \emph{market dynamics} was first introduced in \cite{alvarado2001stability}, where a dynamic pricing signal reflecting a filtered version of system energy imbalance was proposed. Any excess (resp. shortage) of power supplied was viewed as depressing (resp. lifting) its value.
In that setting, generators and loads respond to this signal by adjusting generation and consumption, which in turn changes the energy imbalance and thus affects the price. This work pioneered the study of the grid-market-participant
interplay, yet it did not provide an explicit economic interpretation of this price signal and the corresponding participants' response.   
Since then, follow-up work
has aimed at accommodating more physical constraints 
\cite{jin2013designing,watts2004influence,mota2001dynamic}, including congestion management \cite{mota2001dynamic}, as well as providing control theoretical guarantees, including delays \cite{kiani2010effect} and discrete updates \cite{liang2015stability}.
However, they are still predicated on the similar ad-hoc designs that lack economic guarantees of efficiency and incentive alignment.

Recently, designs for dynamic price signals, based on Lagrange dual gradient algorithms of an economic dispatch problem, have been proposed \cite{jokic2009real,stegink2017unifying,stegink2017frequency,stegink2018hybrid}. These pricing schemes systematically embody a diverse range of operational constraints, and lead to a principled design with a direct economic interpretation of equilibrium prices as well as stability guarantees. 
This Lagrangian-based approach further allows market participation based on both quantity~\cite{jokic2009real,stegink2017unifying} and price bids~\cite{stegink2017frequency,stegink2018hybrid}. 
However, such techniques only allow for a limited homogeneous set of individual behavioral laws -- mapping market outcomes to individual bid updates -- that are analyzed on a case by case basis without systematic guarantee of incentive compatibility.

\subsubsection*{Contributions of this work}

This work builds upon recent saddle-based distributed optimization study to develop a general framework for the design and analysis of market dynamics that account for a wide range of participants' (rational) bidding behavior, market efficiency goals and network operational constraints, while preserving system-wide stability and ensuring incentive compatibility.
More precisely, we consider a setting in which participants receive price and dispatch information from the current market outcome, and seek to maximize their individual utility by updating their bids via a version of dynamic gradient play or best response. 

We identify a particular notion of \emph{alignment}, between participants' bidding behavior and the grid planner's goals, that leads to a systematic design of market dynamics. Such a design is guaranteed to drive the closed-loop system to an equilibrium that (a) is economically efficient and satisfies all operational constraints required by the grid planner; (b) is incentive compatible with all individual participants.
Our alignment condition may be satisfied even when different participants choose different update strategies. 

We show that this alignment condition provides a rational explanation for observations of participants' behavior in previous literature, in both quantity \cite{alvarado2001stability,jokic2009real,stegink2017unifying} and price \cite{stegink2017frequency} bidding settings. 
More specifically, we find that our alignment condition is implicitly satisfied in all of these cases. These results suggest that this framework can provide deeper understanding of previously proposed methods.


Finally, we investigate an exemplar of rational yet misaligned price bidding strategy, and show that the absence of this alignment property can 
lead to unstable behavior. 
We thus propose a more robust version of the proposed market dynamics, based on regularization, that 
recovers asymptotic convergence and desirable steady-state performance of the closed loop. Our solution can be interpreted as an implicit regularization that aims to penalize the system misalignment, thus driving the system closer to alignment. 
We illustrate our results with numerical simulations on the IEEE 39-bus system.

The remainder of the paper is organized as follows. 
Section~\ref{sec:setup} formalizes the grid planner's goals and sets up the general framework. Section~\ref{sec:aligned_MD} defines the notion of alignment and characterizes the resulting systematic design of market dynamics, with applications to the quantity and price bidding settings.
Section~\ref{sec:misaligned_MD} introduces misaligned bidding behavior and highlights the required market modification to restore (approximate) alignment.
Section~\ref{sec:simulation} presents simulation results that validate the theory.  Section~\ref{sec:conclusion} provides conclusions. 


\subsubsection*{Notation}
Let $\mathbb{R}$, $\mathbb{R}_{\ge0}$ and $\mathbb{N}$ be the sets of real, nonnegative real and natural numbers, respectively. For a finite set $\mathcal{H} \subset \mathbb{N}$, its cardinality is denoted as $|\mathcal{H}|$. For a set of scalar variables $\{y_j, j\in\mathcal{H}\}$, its column vector is denoted as $y_{\mathcal{H}}$. The subscript $\mathcal{H}$ might be dropped if the set is clear from the context.
Given vectors $y\in\mathbb{R}^{|\mathcal{H}|}$ and $u\in\mathbb{R}^{|\mathcal{H}|}$, $y\le u$ implies $y_j\le u_j$, $\forall j\in\mathcal{H}$.
We define an element-wise projection $[y]^+_{u}$ where
\beq\label{eq:def_proj}
[y_j]^+_{u_j}=\left\{
\begin{aligned}
	y_j\ ,& \quad \textrm{if}~y_j > 0~\textrm{or}~u_j > 0 \, , \\
	0\ ,&\quad \textrm{otherwise} .
\end{aligned}
\right.
\eeq
This projection is non-expansive in the sense that for any $u^* \ge 0$, the following holds:
\beq\label{eq:nonexpansive}
{[y]_u^+}^T (u-u^*)\le y^T(u-u^*) \ ,
\eeq
since the element-wise projection is active ($[y_j]^+_{u_j}=0$) only when $y_j \le 0$ and $u_j \le 0$, which still implies $[y_j]_{u_j}^+ (u_j - u_j^*) = 0 \le y_j (u_j - u_j^*)$.

For an arbitrary matrix $Y$, $Y^T$ denotes its transpose. If $Y$ is symmetric ($Y=Y^T$), we use $Y \succeq 0$, $Y \succ 0$, $Y \preceq 0$ and $Y \prec 0$ to denote that $Y$ is positive semidefinite, positive definite, negative semidefinite and negative definite, respectively. $Y^{-1}$ is the inverse of $Y$.
$\mathbf{1}$ is a column vector of all 1's.
In an abuse of notation we employ $0$ to denote a column vector, a row vector or a matrix of all 0's if its dimension is clear from the context.
Given a column vector of variables $y\in\mathbb{R}^{|\mathcal{H}|}$, we use two corresponding diagonal matrices $\Gamma^{y}, T^{y} \in \mathbb{R}^{|\mathcal{H}|\times|\mathcal{H}|}$ to denote rate of change and time constants, respectively, with $\Gamma^{y} = {T^{y}}^{-1}$.
Their $j^{\mathrm{th}}$ diagonal elements are denoted as $\gamma_j^y$ and $\tau^y_j$, respectively.


\section{Problem Formulation and Setup}\label{sec:setup}


We consider a continuous-time model for the interaction among grid dynamics, participants' bidding behavior, and the electricity market clearing process, i.e., \emph{market dynamics}. 
We focus on the performance of this coupled system, in terms of steady-state economic dispatch and incentive alignment, from the market perspective, and frequency stability, from the control perspective. 
For simplicity, we restrict active market participants to controllable generators and assume that loads are inelastic, though the framework can be extended to incorporate elastic consumers and prosumers.

In the next subsections, we start with a planner's problem formulation that embodies all the target cross-timescale goals.
We also set up the real-time interactive structure of the coupled system, followed by our models for power network dynamics and rational participants' bidding.



\subsection{Planner's Problem}

We adopt the viewpoint of a grid planner to formalize the cross-timescale economic and frequency control goals in a single problem.
Consider a power network with a connected directed graph $(\mathcal{N},\mathcal{E})$, where $\mathcal{N}:=\{1,2,\dots,|\mathcal{N}|\}$ is the set of nodes and $\mathcal{E}\subset\mathcal{N}\times\mathcal{N}$ is the set of edges connecting nodes. Each node is usually a bus, while each edge describes a connection between two buses, e.g., a transmission line. 
Without loss of generality, we assume there is only one (aggregate) controllable generator at each bus.
We use $(j,k)$ to denote the line from bus $j$ to bus $k$. An arbitrary orientation is applied such that any $(j,k)\in\mathcal{E}$ implies $(k,j)\notin\mathcal{E}$. Each line $(j,k)\in\mathcal{E}$ is endowed with an impedance $z_{jk}$. 
We further define an incidence matrix $C \in \mathbb{R}^{|\mathcal{N}|\times|\mathcal{E}|}$ for the network graph with its element $C_{j,e}=1$ if $e=(j,k)\in\mathcal{E}$, $C_{j,e}=-1$ if $e=(k,j)\in\mathcal{E}$ and $C_{j,e}=0$ otherwise.

We first formalize the primary and secondary control goals. Given a demand vector $d\!:=\!(d_j, j\!\in\!\mathcal{N})\!\in\!\mathbb{R}^{|\mathcal{N}|}$, we adopt a linearized dynamical model for the power network in transient:
\bseq\label{eq:swingdy_initial}%
\begin{align}
\label{eq:swingdy_initial:1}
\dot \theta 	& =   \omega \\
\label{eq:swingdy_initial:2}
M \dot \omega & =  q -d   -D   \omega   - CB  C^T \theta 
\end{align}
\eseq
where $\theta\!:=\!(\theta_j, j\!\in\!\mathcal{N})\!\in\!\mathbb{R}^{|\mathcal{N}|}$ denotes bus phase angles, $\omega\!:=\!(\omega_j, j\!\in\!\mathcal{N})\!\in\!\mathbb{R}^{|\mathcal{N}|}$ denotes bus frequencies, and
$q\!:=\!(q_j, j\!\in\!\mathcal{N})\!\in\!\mathbb{R}^{|\mathcal{N}|}$ denotes generation dispatch.
Here $M:=\diag(M_j,j\in\mathcal{N})\in\mathbb{R}^{|\mathcal{N}|\times|\mathcal{N}|}$ represents the generators' inertia, and $D:=\diag(D_j,j\in\mathcal{N})\in\mathbb{R}^{|\mathcal{N}|\times|\mathcal{N}|}$ summarizes the generators' damping or frequency-dependent demand with $D_j>0$ in general. 
$B:=\diag(B_{jk},(j,k)\in\mathcal{E})\in\mathbb{R}^{|\mathcal{E}|\times|\mathcal{E}|}$ characterizes the sensitivity of each line flow to the phase angle difference between its two end nodes (cf. the appendix of \cite{you18stabilization} for the derivation of $B$).
Standard assumptions are made to linearize the power flow equations in \eqref{eq:swingdy_initial:2}:
1) bus voltage magnitudes are fixed constant; 2) lines are lossless; 3) reactive power is ignored \cite{kundur1994power}.
Note that the linearized network model \eqref{eq:swingdy_initial} implicitly assumes that the variables $\theta,\omega,q$ as well as the parameter $d$ are \emph{deviations} from their nominal values.
Therefore, $\omega =0$ represents the nominal frequency (\emph{the goal of secondary control}), while $\dot \omega =0$ implies stabilized frequencies (\emph{the goal of primary control}).


It will be convenient for the analysis to define $\tilde {\theta}:=  C^T \theta \in\mathbb{R}^{|\mathcal{E}|}$ as the phase angle differences that determine (deviations of) line flows $B\tilde \theta$ across the network, and rewrite the swing dynamics \eqref{eq:swingdy_initial} in the form:
\bseq\label{eq:swingdynamics}%
\begin{align}
\label{eq:swingdynamics:1}
\dot{ \tilde \theta }	& =  C^T  \omega \\
\label{eq:swingdynamics:2}
M \dot \omega & =  q -d   -D   \omega   - CB  \tilde \theta 
\end{align}
\eseq

We then introduce a canonical tertiary control problem that seeks to find an optimal economic dispatch of generation.  
At the steady state of the power network, the stationary frequencies at the nominal value, i.e., $\omega=0$ and $\dot \omega =0$, render \eqref{eq:swingdynamics:2} a characterization of the nodal power balance over the network:
\beq\label{eq:nodal_power_balance}
 q  - d    - C B\tilde \theta =0 \ .
\eeq
We further impose the lower and upper thermal limits $\underline{F}$ and $\overline{F}$ on the (deviations of) line flows:
\beq\label{eq:line_thermal_limits}
 \underline{F} \le B \tilde  \theta \le  \overline{F}  \ .
\eeq
Then the tertiary control (economic dispatch) problem that minimizes the aggregate generation cost to meet the demand over the network is given by
\begin{subequations}
	\begin{eqnarray}
	\label{eq:edp.a}
	\min_{q,  \theta} &&      \mathbf{1}^T  J(q)    \\
	\label{eq:edp.b}
	\mathrm{s.t.} &&  \eqref{eq:nodal_power_balance}, \eqref{eq:line_thermal_limits}  
	\end{eqnarray}	\label{eq:edp}%
\end{subequations}
where $J(q):=(J_j(q_j),j\in\mathcal{N})$ is a column vector-valued function, with $J_j(\cdot):\mathbb{R}\mapsto \mathbb{R}$ representing the cost function of generator $j$.\footnote{Here $J_j(\cdot)$ is defined on the deviation of generation. For convenience of the analysis, we ignore generation capacity constraints. An alternative is to incorporate them in properly redesigned cost functions \cite{mallada2017optimal}.}
We assume that $J_j(\cdot)$ is \emph{strictly convex and twice differentiable}. 


The problem \eqref{eq:edp} can be expressed compactly in terms of $q$ by first rewriting the nodal power balance \eqref{eq:nodal_power_balance} in terms of the network power balance
\beq
\mathbf{1}^T \left( q- d -CB \tilde \theta \right) = \mathbf{1}^T  \left( q- d \right) = 0 \ ,
\eeq
where the first equality follows from $\mathbf{1}^T C = 0$.
Then defining the weighted \emph{Laplacian} matrix of the power network as $L:=CBC^T$ enables \eqref{eq:nodal_power_balance} to be reorganized as
\beq
q  - d = CB\tilde \theta = CBC^T\theta  = L \theta \ .
\eeq
The line flows $B\tilde \theta$ can accordingly be expressed in terms of the bus net power injections $q-d$:
\beq
B\tilde \theta = BC^T L^{\dagger} (q-d) \ , 
\eeq
where $L^{\dagger}$ denotes the Moore-Penrose inverse of $L$. 
Here $BC^T L^{\dagger}$ is the \emph{power injection shift matrix} of the power network.
We further let $H^T:= [(BC^T L^{\dagger})^T, -(BC^T L^{\dagger})^T]^T \in \mathbb{R}^{2|\mathcal{E}| \times |\mathcal{N}|}$ and $F:=[\overline F^T , -\underline F^T ]^T$ be the stacked shift matrix and thermal limit vector, respectively. 
The resulting equivalent reformulation of the tertiary control problem \eqref{eq:edp} is then
\begin{subequations}
	\begin{eqnarray}
		\label{eq:edp2.a}
		\min_{q} &&      \mathbf{1}^T  J(q)    \\
		\label{eq:edp2.b}
		\mathrm{s.t.} &&  
		\mathbf 1^T ( q-d )   = 0   \ : \ \lambda   \\
		\label{eq:edp2.c}
		&&	 H^T(q-d) \le {F}  \ : \ \eta \ge 0
	\end{eqnarray}	\label{eq:edp2}%
\end{subequations}
where the Kirchhoff's Laws are embedded in the matrix $H$, and $\lambda \in\mathbb{R}$,  $\eta \in\mathbb{R}_{\ge0}^{2|\mathcal{E}|}$ are the respective Lagrange dual variables for \eqref{eq:edp2.b}, \eqref{eq:edp2.c}.
The equivalence between the  formulations \eqref{eq:edp} and \eqref{eq:edp2} is formally stated below (proof by contradiction).
\begin{lemma}\label{lm:edp}
    $(q^*, \tilde \theta^*)$ is an optimal solution to \eqref{eq:edp} if and only if $\tilde \theta^* = C^T L^{\dagger} (q^*-d)$ holds and  $q^*$ is an optimal solution to \eqref{eq:edp2}.
\end{lemma}

The optimal primal-dual solution $(q^*,\lambda^*,\eta^*)$ to \eqref{eq:edp2} leads to an implicit definition of clearing prices based on the dual optimizers ($\lambda^*,\eta^*$)~\cite{cai2017distributed}, given by the locational marginal prices (LMPs) $\lambda^*\cdot \mathbf{1} - H\eta^*$.
As guaranteed by 
the KKT conditions of \eqref{eq:edp2},
they are incentive compatible for individual generators in the sense that the clearing prices match the marginal generation costs, i.e., $\nabla J(q^*) = \lambda^*\cdot \mathbf{1} - H\eta^*$, where $\nabla J(q):=(\nabla J_j(q_j),j\in\mathcal{N})$ is a column vector-valued function of element-wise increasing gradients $\nabla J_j(q_j)$.

Adding now the primary and secondary control goals to the tertiary control problem leads to the following planner's problem:

\noindent
\textbf{Planner's problem}
\begin{subequations}\label{eq:planner}
	\begin{eqnarray}
		\label{eq:uedp.a}
		\min_{p,q,\omega,\tilde\theta} &&      \mathbf{1}^T  J(p)  + \frac{1}{2} \omega^T D \omega   \\
		\label{eq:uedp.b}
		\mathrm{s.t.} &&     q=p \ : \ \alpha \\
				\label{eq:uedp.c}
		&& \mathbf 1^T ( q-d )   = 0   \ : \ \lambda  \\
		\label{eq:uedp.d}
		&&	 H^T(q-d) \le {F}  \ : \ \eta\ge 0  	\\
				\label{eq:uedp.e}
		&& q-d -D\omega -CB\tilde\theta =0  \ : \ \nu 
	\end{eqnarray}	\label{eq:uedp}%
\end{subequations}
where $p:=(p_j,j\in\mathcal{N})\in\mathbb{R}^{|\mathcal{N}|}$ (used also to denote quantity bids) here represents more broadly individual output scheduling of generators, and is required to align with market dispatch through \eqref{eq:uedp.b}. 
Note that we abuse the notation to define the Lagrange dual variables $\alpha \in\mathbb{R}^{|\mathcal{N}|}$ (used also to denote price bids), $\lambda \in\mathbb{R}$, $\eta \in\mathbb{R}_{\ge0}^{2|\mathcal{E}|}$ and $\nu \in\mathbb{R}^{|\mathcal{N}|}$ for \eqref{eq:uedp.b}-\eqref{eq:uedp.e}, respectively.

All the cross-timescale control goals are implicitly embedded in the optimum of the planner's problem. This fact is characterized in the following theorem (proof in Appendix~\ref{apx:proof_of_teo_planner}).
\begin{theorem}\label{teo:planner}
$(p^*,q^*,\omega^*,\tilde \theta^*)$ is an optimal solution to \eqref{eq:uedp}
if and only if $p^*=q^*$, $\omega^*=0$ and $\tilde \theta^* = C^T L^{\dagger} (q^*-d)$ hold and $q^*$ is an optimal solution to \eqref{eq:edp2}.
\end{theorem}

\begin{corollary}
The optimum of the planner's problem realizes
\bi
\item (primary control) frequency stabilization $\dot \omega =0$, i.e., \eqref{eq:uedp.e};
\item (secondary control) the nominal frequency $\omega^* =0$;
\item (tertiary control) the economic dispatch $q^*$ of \eqref{eq:edp2};
\item (compatible participants' incentives) fully incentivized generation $\nabla J(p^*) = \lambda^*\cdot \mathbf{1} - H \eta^* $.
\ei
\end{corollary}
\noindent
The last statement follows from the KKT conditions of \eqref{eq:uedp}, in particular $\nu^* =\omega^* =0$.
Indeed, the planner's problem \eqref{eq:uedp} suggests another way of defining clearing prices based on dual optimizes $(\lambda^*,\eta^*,\nu^*)$, given by $\lambda^*\cdot \mathbf{1} - H \eta^* - \nu^*$, such that they are incentive compatible with the economic dispatch $q^*$:
\beq\label{eq:planner's_pricing}
\nabla J(q^*) = \lambda^*\cdot \mathbf{1} - H \eta^* - \nu^* \ .
\eeq
This means of pricing essentially boils down to the canonical LMPs $\lambda^*\cdot \mathbf{1} - H \eta^* $ with $\nu^*=0$.

Central to our developments will be the Lagrangian for the convex planner's problem \eqref{eq:uedp}, i.e., 
\begin{small}
\beq
 \begin{aligned}\label{eq:base_Lagrangian}
		 & L(p,q,\omega,\tilde \theta, \alpha,\lambda,\eta,\nu )\\
		 := \; &     \underbrace{  \mathbf{1}^T  J(p)}_{\textrm{generators}} +  \underbrace{\frac{1}{2} \omega^T D\omega \!+\! \nu^T\left( q\!-\!d\!-\!D\omega \!-\!CB\tilde\theta \right) }_{\textrm{network}} \\ 
		   & \!+\!\!\!\! \underbrace{\alpha^T(q\!-\!p)}_{\textrm{generators or market}}
		  \!\!\!\! - \underbrace{\lambda \cdot   \mathbf 1^T (q\!-\!d)   \!+\! \eta^{T} \left(  H^T(q\!-\!d) \!-\! F  \right) }_{\textrm{market}}  \ ,
\end{aligned}
\eeq
\end{small}%
which we refer to as the planner's Lagrangian, with the potential responsible party for each term.
It is easy to see that \eqref{eq:base_Lagrangian} is convex in the primal variables $(p,q,\omega,\tilde \theta)$ and concave (linear) in the dual variables $(\alpha,\lambda, \eta,\nu)$.
The KKT conditions establish a bijective mapping between a min-max saddle point $(p^*,q^*,\omega^*,\tilde \theta^*,\alpha^*,\lambda^*, \eta^*\ge 0,\nu^*)$ of the planner's Lagrangian \eqref{eq:base_Lagrangian} and an optimal primal-dual solution to the planner's problem \eqref{eq:uedp}.
We further refer to a function as a \emph{reduced} planner's Lagrangian, if it is the optimum of the planner's Lagrangian \eqref{eq:base_Lagrangian} over a subset of the primal and dual variables. 

\subsection{Real-Time Interactive Structure}

In practice, the planner's problem \eqref{eq:planner}
is not implementable due to the lack of knowledge of generators' cost functions. This poses significant challenges for the grid planner to realize economic dispatch in an incentive-compatible manner. 
To overcome this obstacle, we propose to use the real-time interaction among the grid, market, and participants to automatically achieve all the economic and frequency control goals.
We thus consider a continuous-time setting and investigate two classes of dynamic bidding mechanisms, based on quantity and price, that allow each generator to determine its own bid while simultaneously allowing the market to update its prices and dispatch.

We propose a unified framework for the grid-market-participant loop, with a schematic layout of the interactive structure shown in Fig.~\ref{fig:interactive_structure}. 
The power network shares real-time bus frequencies ($\omega$) and (given) inelastic demand ($d$) with the market. The market determines the signals of clearing prices ($\pi$) and generation dispatch ($q$) based on the network conditions and generators' bids.
The bids can take the form of a quantity ($p$), suggesting the desired output of a generator, or a price ($\alpha$), indicating the desired unit price of a generator for its output. 
Each individual generator responds to the market signals by implementing the prescribed 
dispatch and updating its own bid to reflect its preference to the market. 
The dispatch immediately implemented affects the power network dynamics and may interfere with the frequency stability.

\begin{figure}[ht]
    \centering
    \includegraphics[width=222pt, height=131pt]{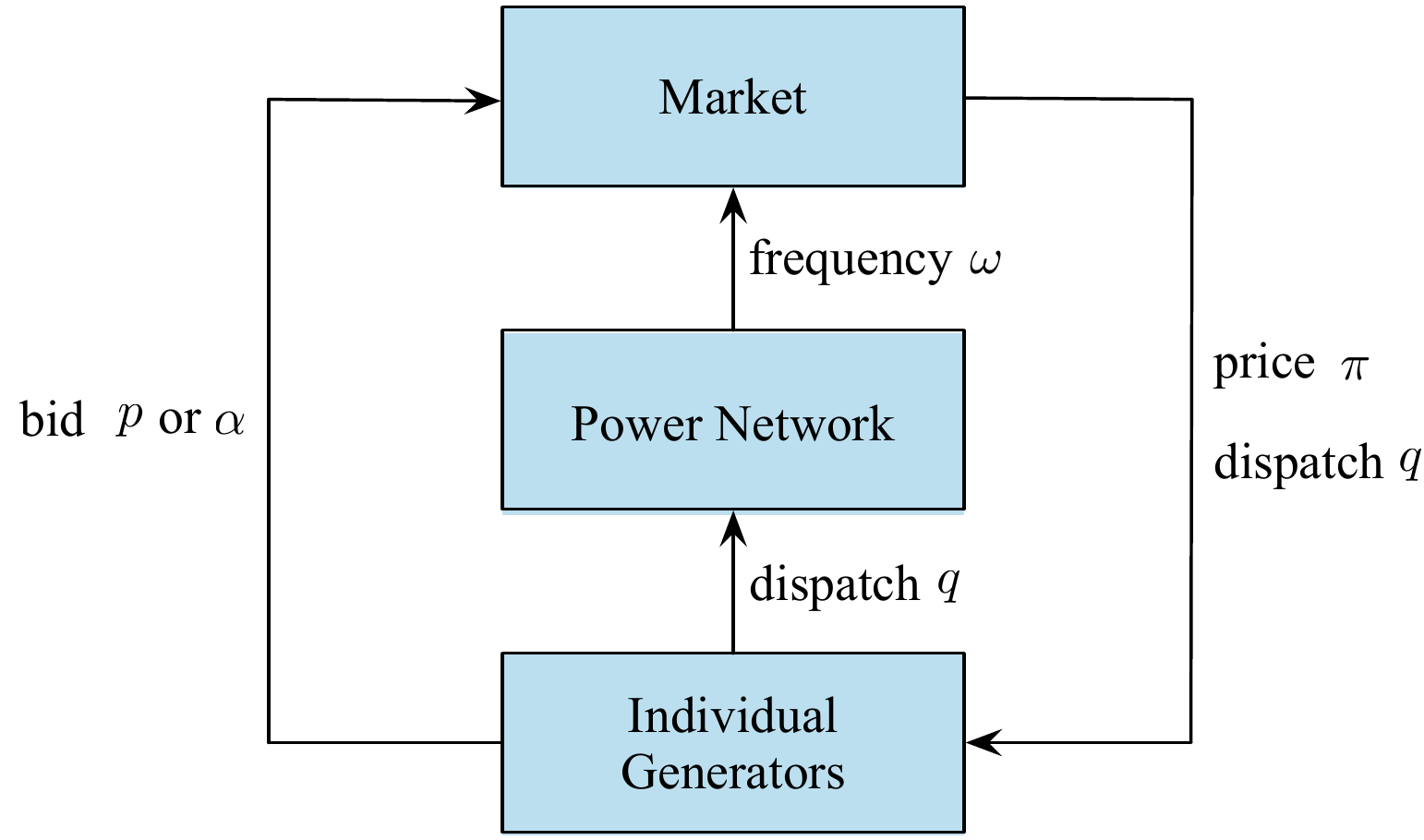}
    \caption{Interactive structure. Only time varying exchanged information is depicted.}
    \label{fig:interactive_structure}
\end{figure}

Inspired by \eqref{eq:planner's_pricing} from the planner's problem \eqref{eq:uedp}, we then define transient clearing prices $\pi$ in Fig.~\ref{fig:interactive_structure} to generalize the canonical LMPs as follows.
\begin{definition}
Market clearing prices are defined as
	\beq\label{eq:DLMP}
	 	\pi := \lambda \cdot \mathbf{1} - H\eta -\nu \ .
	\eeq
\end{definition}


\noindent
Intuitively, $\lambda$ prices global network power imbalance, $H \eta$ prices line congestion, and $\nu$ prices local bus power imbalance. 
This dynamic version of LMPs can be interpreted as \emph{transient shadow prices}, and are bus-dependent and time-varying, embodying the sufficient control authority of the market to react to the changing network operational conditions and participants' bids.

The fundamental challenge for the coupled system in Fig.~\ref{fig:interactive_structure} to simultaneously realize the primary, secondary and tertiary control, with a particular guarantee for compatible participants' incentives, now boils down to two problems. 
First, \emph{How to characterize individual bidding behavior?}
Second, \emph{How to design market control laws of pricing and dispatch?}
We aim to identify their underlying connections and address them systematically.

\subsection{Rational Bidding of Individual Generators}

We provide a principled approach to the first problem by capturing the incentives of individual participants, which are assumed to be rational price takers.
A rational generator $j\in\mathcal{N}$ can be modeled to pursue an input-output optimization of a bidding problem, parameterized by market dispatch $q_j$ and price $\pi_j$ (inputs), that decides its bid for quantity $p_j$ or price $\alpha_j$ (outputs). Therefore, we formulate a general form of the bidding problem for each generator $j$ as

\noindent
\textbf{Individual generator bidding problem}\\
input: \emph{clearing price $\pi$ and dispatch $q$}; output: \emph{bid $p,\alpha$}
\bseq\label{eq:individual_problem_general}
\bq
\max_{p_j} && U_j(p_j;q_j,\pi_j) \\
\mathrm{s.t.} && G_j(p_j;q_j,\pi_j) \le  0 \ : \  \alpha_j
\eq
\eseq
where $U_j(\cdot)$ and $G_j(\cdot)$ are respectively real-valued concave and convex functions to represent the objective and constraint of generator $j$. 
We use the dual variable $\alpha_j$, a proxy for the marginal cost, as the price bid. We will provide more intuitions later.
Note that only one type of bid $p$ or $\alpha$ will be present and the general expression describing inequality constraints also captures equality constraints.

We adopt a gradient-based methodology to characterize the rational bidding behavior of individual generators. 
To solve the bidding problem \eqref{eq:individual_problem_general}, we define its Lagrangian $L_j(p_j,\alpha_j;q_j,\pi_j)$ and choose between
\bseq\label{eq:gen_gradient_general}
\begin{align}
    \dot p_j & =  \gamma_j^p\cdot \nabla_{p_j} L_j  \quad \quad \quad (\textrm{resp.~} \dot \alpha_j  =  -\gamma_j^{\alpha} \cdot \nabla_{\alpha_j } L_j ) 
    \label{eq:gen_gradient_general.a}\\
    p_j &=  \arg \max_{p_j} L_j     \quad (\textrm{resp.~} \alpha_j  =  \arg \min_{\alpha_j} L_j )
    \label{eq:gen_gradient_general.b}
\end{align}
\eseq
as the individual dynamic gradient play \cite{shamma2005fictitious} or best response for bidding.
We will discuss next the explicit formulations of the bidding problem \eqref{eq:individual_problem_general} and the resulting bidding behavior, and further develop the countermeasure of market control laws to address the second problem.


\section{Aligned Market Dynamics}\label{sec:aligned_MD}

In this section, we develop a systematic design of market dynamics that is able to accommodate a family of participants' bidding behavior, referred to as \emph{aligned bidding}. 
The design is inspired by an formulation of a network problem based on the swing dynamics \eqref{eq:swingdynamics} as well as its connection with the planner's Lagrangian \eqref{eq:base_Lagrangian}.
We will first characterize the general paradigm of alignment, and then showcase two specific examples of such market dynamics designs from existing literature.

\subsection{Network Problem}

We first formulate a network problem that implicitly embodies the physical swing dynamics \eqref{eq:swingdynamics}:

\noindent
\textbf{Network problem}\\
input: \emph{dispatch $q$}; output: \emph{frequencies $\omega$}
\bseq\label{eq:network_problem}
\bq
\min_{\omega,\tilde \theta} && \frac{1}{2} \omega^T D \omega \\
\mathrm{s.t.} && q-d-D\omega -CB\tilde \theta = 0  \ : \   \nu 
\eq
\eseq
By defining its Lagrangian
\beq
L_n(\omega,\tilde \theta,\nu): = \frac{1}{2} \omega^T D \omega + \nu^T(q-d-D\omega -CB\tilde \theta) \ ,
\eeq
we can express the swing dynamics \eqref{eq:swingdynamics} as
\bseq\label{eq:swing_dynamics_Lagrangian_interpretation}
\begin{align}
\label{eq:swing_dynamics_Lagrangian_interpretation.a}
    \omega \ = \ & \arg \min_{\omega} L_n    \\
\label{eq:swing_dynamics_Lagrangian_interpretation.b}
    \dot{\tilde \theta} \ = \ &  -\Gamma^{\tilde \theta} \nabla_{\tilde \theta} L_n \\
    \label{eq:swing_dynamics_Lagrangian_interpretation.c}
    \dot \nu  \ = \ & \Gamma^{\nu} \nabla_{\nu} L_n
\end{align}
\eseq
with $\Gamma^{\tilde \theta} = B^{-1}$ and $\Gamma^{\nu} = M^{-1}$. Note that \eqref{eq:swing_dynamics_Lagrangian_interpretation.a} enforces $\omega \equiv \nu $ even in transient, which \emph{allows us to use $\omega \leftrightarrow \nu$ interchangeably}. Therefore, the clearing prices $\pi$ in \eqref{eq:DLMP} are equivalently
\beq
\pi = \lambda \cdot \mathbf{1} - H\eta -\omega \ .
\eeq

Note that the network problem \eqref{eq:network_problem} can be viewed as part of the planner's problem \eqref{eq:uedp}.
More formally, we define the relation between the network and the grid planner as \emph{aligned} below.
\begin{definition}\label{def:network_alignment}
The power network dynamics \eqref{eq:swingdynamics} is aligned with the grid planner's goals in the sense that
\beq\label{eq:network_alignment}
\nabla_u L_n = \nabla_u L
\eeq
holds for the network variables $u = \omega, \tilde\theta , \nu$.
\end{definition}
The alignment \eqref{eq:network_alignment} means that  \eqref{eq:swing_dynamics_Lagrangian_interpretation} can be equivalently expressed in terms of the planner's Lagrangian $L$ in \eqref{eq:base_Lagrangian}.
It further implies that the physical response of the power network automatically implements a partial saddle flow \eqref{eq:network_alignment} of $L$ or its possibly reduced variant.
Similar interpretations are identified in \cite{zhao2016unified,li2016connecting,mallada2017optimal,wang2017distributed1,wang2017distributed2}.
Since the network chooses \eqref{eq:swing_dynamics_Lagrangian_interpretation.a} for $\omega$ (reduction), we accordingly define a reduced planner's Lagrangian
\beq\label{eq:reduced_planner's_Lagrangian_bar_L}
    \bar L (p,q,\tilde \theta, \alpha,\lambda,\eta, \nu) \ := \  \min_{\omega} \ L(p,q,\omega,\tilde\theta, \alpha, \lambda,\eta,\nu)   \ .
\eeq

\subsection{Aligned Bidding and Market Dynamics Design}

We model the individual bidding behavior of generator $j\in\mathcal{N}$ by the dynamic gradient play or best response in \eqref{eq:gen_gradient_general} with respect to a Lagrangian $L_j$ of its bidding problem. In light of the same gradient-based structure as \eqref{eq:network_alignment}, a particular family of participants' behavior is identified to satisfy the following notion of \emph{alignment}.
\begin{definition}\label{def:alignment}
Individual generators' bidding behavior is aligned with the grid planner's goals if 
\beq\label{eq:gen_gradient_alignment_matched}
\nabla_u L_j  =  \nabla_u \tilde L
\eeq
holds for the bidding variables $u=p_j,\alpha_j$,
where $\tilde L$ is a possibly reduced variant of $\bar L$ in \eqref{eq:reduced_planner's_Lagrangian_bar_L}.
\end{definition}
We will make the following assumption on $\tilde L$:
\begin{assumption}\label{ass:finiteness}
$\tilde L$ is a finite function in its domain.
\end{assumption}

\begin{remark}
In practice the value of $\tilde L$, or any other reduced planner's Lagrangians, may be infinite after the reduction. This circumstance usually arises when the planner's Lagrangian $L$ is linear with respect to the variables to be optimized. 
Assumption~\ref{ass:finiteness} basically does not allow such variables to be optimized alone in the reduction.
An alternative to overcoming this limitation is to introduce an additional regularization term of the form of $\Vert x- \hat x\Vert^2$, with $\hat x$ being an auxiliary variable, whenever optimizing $x$ leads to infinite values~\cite{goldsztajn2019proximal,you2020saddle}. Therefore, we will implicitly assume all the reduced Lagrangians discussed later to be finite.
\end{remark}

This alignment between participants and the grid planner in Definition~\ref{def:alignment} also connects the way each individual participant bids \eqref{eq:gen_gradient_general} with a partial saddle flow of $\tilde L$ (or its possibly reduced variant), in addition to \eqref{eq:network_alignment} realized by the swing dynamics \eqref{eq:swingdynamics}.
Note that the saddle points of the planner's Lagrangian $L$ in \eqref{eq:base_Lagrangian}, and thus also $\tilde L$, optimally solve the planner's problem \eqref{eq:uedp} and achieve all of its goals. This connection inspires a design of aligned market dynamics that complements the saddle flow through pricing and dispatch.


In general, with either dynamic bidding mechanism of quantity $p$ or price $\alpha$, the market seeks to solve an input-output variant of the planner's problem \eqref{eq:uedp}, where some quantities, e.g., bids and frequencies, are assumed as given (inputs), and prices and dispatch are to be computed and released to participants (outputs). We formulate such a problem in a generic form as follows:

\noindent
\textbf{Market problem}\\
input: \emph{bids $\alpha,p$ and frequencies $\omega$}; output: \emph{clearing prices $\pi$ and dispatch $q$}
\bseq\label{eq:market_problem_general}
\bq
\min_{q} && U_m(q;p,\alpha,\omega) \\
\mathrm{s.t.} && G_m(q;p,\alpha,\omega) \le 0 \ : \  (\lambda,\eta\ge 0) 
\eq
\eseq
where $U_m(\cdot)$ and $G_m(\cdot)$ are real- and vector-valued convex functions that respectively represent the market objective and constraints.  

We could also develop gradient-based market control laws for pricing and dispatch that can be interpreted as a primal-dual algorithm to solve the market problem \eqref{eq:market_problem_general}.
In particular, we define for the market problem \eqref{eq:market_problem_general} its Lagrangian $L_m(q,\lambda,\eta;p,\alpha,\omega)$ and likewise choose between
\bseq\label{eq:market_gradient_general}
\begin{align}
    \dot u & = - \Gamma^u \nabla_u L_m  \quad\quad~ (\textrm{resp.~} \dot u  =  \Gamma^u \left[ \nabla_u L_m \right]^+_{\eta} )
    \label{eq:market_gradient_general.a}\\
    u &=  \arg \min_u L_m \quad (\textrm{resp.~}  u =  \arg \max_{u:\eta\ge 0} L_m)
    \label{eq:market_gradient_general.b}
\end{align}
\eseq
for any primal variable $u=q$ (resp. dual variable $u=\lambda,\eta$) to solve for its optimal value.
The projection $[\cdot]^+_{\eta}$ applies only to $u=\eta$, which guarantees that the trajectory of $\eta(t)$ starting from an arbitrary non-negative point remains non-negative. 

Given the alignment of both the power network \eqref{eq:network_alignment} and participants \eqref{eq:gen_gradient_alignment_matched}, the key design is to exploit $\tilde L$ and the choice of participants' bidding update (dynamic gradient play or best response) in \eqref{eq:gen_gradient_alignment_matched}, and extract the corresponding market problem \eqref{eq:market_problem_general} from the planner's problem \eqref{eq:uedp}.
In particular, suppose $\tilde L$ is the optimum of $\bar L$ over a subset of (reduced) variables $v_m$, i.e., $\tilde L = \bar L\vert_{v_m^*}$ with $v_m^*$ being the optimizer, and the bidding of each generator $j$ involves \eqref{eq:gen_gradient_general.b} for a subset of its variables $v_j$. We propose to design the market such that
\begin{enumerate}
    \item for the market variables $u= q,\lambda,\eta$,
    \beq\label{eq:market_alignment}
    \nabla_u L_m = \nabla_u  \bar L\vert_{v_j^*} 
    \eeq
    holds (thus also aligned), where $v_j^*$ is the optimizer;
    \item \eqref{eq:market_gradient_general.b} is chosen for $v_m$ to yield $\tilde L$ in \eqref{eq:gen_gradient_alignment_matched}.
\end{enumerate}
Based on the bidding mechanism, we can always apply optimization decomposition to the planner's problem \eqref{eq:uedp} to obtain the desired market problem \eqref{eq:market_problem_general} (see examples in Section~\ref{ssec:ill_examples}). 
By this means, the market control laws \eqref{eq:market_gradient_general} are basically designed to complement the saddle flow of a particular reduced planner's Lagrangian that accounts for all the reduced parts, given by
\beq\label{eq:reduced_Lagrangian_final}
\hat L : =\bar L\vert_{v^*_m,v^*_j,j\in\mathcal{N}} \ .
\eeq


Such aligned market dynamics render the joint dynamics of the grid, market, and participants a (projected) saddle flow of $\hat L$. 
Thus, the grid-market-participant loop can be expressed as
\beq\label{eq:grid-market-participant_loop}
\begin{bmatrix}
	T^z & \\
	& T^\sigma   
\end{bmatrix}
\begin{bmatrix}
    \dot z \\
    \dot \sigma
\end{bmatrix}
= 
\begin{bmatrix}
   - \nabla_z \hat L(z,\sigma)  \\
    \left[\nabla_{\sigma} \hat L(z,\sigma) \right]^+_{\eta}
\end{bmatrix}  \ ,
\eeq
where $z$ and $\sigma$ are subsets of the respective primal variables $(p,q,\tilde\theta)$ and the dual variables $(\alpha,\lambda,\eta, \nu)$, which are updated using the gradient information via \eqref{eq:gen_gradient_general.a}, \eqref{eq:swing_dynamics_Lagrangian_interpretation.b}, \eqref{eq:swing_dynamics_Lagrangian_interpretation.c} and \eqref{eq:market_gradient_general.a}. 
We slightly abuse the notation such that the projection $[\cdot]^+_{\eta}$ only applies to part of the gradient corresponding to $\eta$ in $\sigma$. 
The rest variables are updated based on \eqref{eq:gen_gradient_general.b}, \eqref{eq:swing_dynamics_Lagrangian_interpretation.a} and \eqref{eq:market_gradient_general.b}.


To gain insights into this closed-loop interaction \eqref{eq:grid-market-participant_loop}, we next show that its equilibria correspond to optimal solutions to the planner's problem \eqref{eq:uedp}. We further show that it converges asymptotically to one such equilibrium point under mild conditions.
If $\eta$ is contained in $\sigma$, define
\beq\label{eq:initial_set}
	\mathbb{I}:=\bigg\{(z,\sigma) \ | \      \eta\in\mathbb{R}_{\ge0}^{2|\mathcal{E}|} \bigg\}
\eeq
as the set of initial points in order to guarantee a non-negative trajectory for $\eta$; otherwise, define $\mathbb{I} := \mathbb{R}^{|z|+|\sigma|}$.
Then the equilibrium of the grid-market-participant loop \eqref{eq:grid-market-participant_loop} can be characterized by the following theorem (proof in Appendix~\ref{apx:proof_of_teo_eqm}):
\begin{theorem}\label{teo:eqm}
    Let Assumption~\ref{ass:finiteness} hold.
	For the grid-market-participant loop \eqref{eq:grid-market-participant_loop}, a point $(z^*,\sigma^*)\in\mathbb{I}$ is an equilibrium if and only if $(z^*,\sigma^*)$ corresponds to an optimal primal-dual solution to the planner's problem \eqref{eq:uedp}.
\end{theorem}

Theorem~\ref{teo:eqm} indicates that each equilibrium point not only restores the nominal frequency, but also achieves underlying economic dispatch - demand met in an economically efficient manner and line thermal limits respected - while fulfilling compatible participants' incentives through individual bidding.


We proceed to show the convergence of the closed-loop system \eqref{eq:grid-market-participant_loop} to one equilibrium point. Given the initial condition of $\mathbb{I}$,
define
\beq\label{eq:equilibrium_set}
\mathbb{E}:=\left\{(z,\sigma) \ | \  \dot z,\dot \sigma= 0 \right\}
\eeq
as the set of its equilibrium points. We make the following assumption on the system observability that leads to the asymptotic stability of the equilibrium set described in Theorem \ref{teo:stability} (proof in Appendix~\ref{apx:proof_of_teo_stability}).

\begin{assumption}[Observability]\label{ass:observability}
The grid-market-participant loop \eqref{eq:grid-market-participant_loop} has an observable output such that for any of its trajectories $(z(t),\sigma(t))$ that satisfy $\hat L(z^*,\sigma(t)) \equiv \hat L(z^*,\sigma^*)$ and $\hat L(z(t),\sigma^*) \equiv \hat L(z^*,\sigma^*)$, we have $\dot z, \dot \sigma \equiv 0$.
\end{assumption}

\begin{theorem}\label{teo:stability}
	If Assumptions~\ref{ass:finiteness} and \ref{ass:observability} hold, the equilibrium set $\mathbb{E}$ is globally asymptotically stable on $\mathbb{I}$. In particular, starting from any initial point in $\mathbb{I}$, a trajectory $(z(t),\sigma(t))$ of the grid-market-participant loop \eqref{eq:grid-market-participant_loop} remains bounded for $t\ge 0$ and converges to $(z^*,\sigma^*)$ with $t \rightarrow \infty$, where $(z^*,\sigma^*)$ is one specific equilibrium point in $\mathbb{E}$.  
\end{theorem} 


Theorem~\ref{teo:stability} implies that, under mild observability conditions, the aligned participants' bidding behavior together with the proposed market dynamics can essentially function as a feedback controller on the power network dynamics to simultaneously realize the cross-timescale primary, secondary and tertiary control with compatible participants' incentives. 
Indeed, this notion of alignment establishes connections to the rationale of saddle flow dynamics~\cite{you2020saddle,goldsztajn2019proximal,dhingra2018proximal} that underlies such principled grid-market-participant loop \eqref{eq:grid-market-participant_loop}. 



\subsection{Illustrative Examples}\label{ssec:ill_examples}

We next present two examples of aligned market dynamics, based on the two common bidding mechanisms of quantity and price, respectively. 
The two examples take root in existing literature \cite{stegink2017unifying,stegink2017frequency,you18stabilization}.
For illustration purposes we follow a uniform choice of bidding behavior  (either \eqref{eq:gen_gradient_general.a} or \eqref{eq:gen_gradient_general.b} for all participants) as in~\cite{stegink2017unifying,you18stabilization}. We note, however, the results presented in this paper can accommodate mixed choices, thus generalizing previous works.
We also show through the examples that the observability condition in Assumption~\ref{ass:observability} is readily satisfied for the grid-market-participant loop \eqref{eq:grid-market-participant_loop}.


\subsubsection{Quantity Bidding}

The quantity bidding mechanism allows individual generator participants to bid their own decisions on the amount of electricity generation $p$ into the market. All the accepted quantity bids $p$ are respected and taken as the generation dispatch $q$. However, the market aims to match the supply with given demand and meet the network operational constraints by setting appropriate clearing prices $\pi$ to coordinate these quantity bids. 

Since the dispatch $q_j$ of generator $j$ follows its quantity bid $p_j$, it is able to maximize its profit from the market, given the clearing price $\pi_j$ at bus $j\in\mathcal{N}$, by solving its bidding problem 
\begin{subequations}\label{eq:gen_quantity_bidding_problem}
\begin{eqnarray}
		\label{eq:gen_quantity_bidding_problem.a}
		U_j(p_j;\pi_j)  & := &      \pi_j p_j - J_j(p_j)    \\
		\label{eq:gen_quantity_bidding_problem.b}
		G_j(p_j;\pi_j) \le 0  & : &  \varnothing
	\end{eqnarray}	
\end{subequations}
The unconstrained problem \eqref{eq:gen_quantity_bidding_problem} immediately suggests
\beq\label{eq:L_j_quantity}
L_j(p_j;\pi_j) \ : = \  \pi_j p_j - J_j(p_j) \ ,
\eeq
and the dynamic gradient play of \eqref{eq:gen_gradient_general.a} on $p_j$ gives the bidding strategy of generator $j$ \cite{stegink2017unifying,you18stabilization}:
\begin{align}\label{eq:gradient_play_gen_quantity_bidding}
\tau_j^p \dot p_j  \  =	\      \pi_j  - \nabla J_j(p_j) \ , \quad j\in\mathcal{N} \ .
\end{align}
Note that \eqref{eq:gradient_play_gen_quantity_bidding} reveals a direct economic interpretation that generator $j$ tends to augment production $p_j$ if the offered clearing price $\pi_j$ exceeds its marginal cost $\nabla J_j(p_j)$; otherwise, it will curtail production $p_j$. It is therefore incentivized to adapt its production $p_j$ in a way that matches its marginal cost $\nabla J_j(p_j)$ with the clearing price $\pi_j$, which will bring it the maximum profit.
Therefore, \eqref{eq:gradient_play_gen_quantity_bidding} represents exactly the rational quantity bidding behavior of a price-taking generator.

Further, \eqref{eq:gradient_play_gen_quantity_bidding} satisfies the alignment condition in Definition~\ref{def:alignment} based on the following reduced planner's Lagrangian:
\beq\label{eq:aligned_participant_reduced_Lagrangian_quantity}
\tilde L(p,\tilde \theta, \lambda,\eta,\nu) : = \min_{q}\max_{\alpha} \ \bar L(p,q,\tilde \theta, \alpha,\lambda,\eta,\nu ) \ ,
\eeq
which leads to $q\equiv p$ even in transient, i.e., dispatch follows quantity bids, 
exactly as the bidding mechanism mandates.

Given \eqref{eq:gradient_play_gen_quantity_bidding} and \eqref{eq:aligned_participant_reduced_Lagrangian_quantity}, $v_m$ contains $q$ and $\alpha$, while $v_j$ is $\varnothing$, $j\in\mathcal{N}$. Based on \eqref{eq:market_alignment}, we aim for $L_m$ such that $\nabla_u L_m = \nabla_u \bar L$ holds for $u=q,\lambda,\eta$. 
Note that now the market enforces $q=p$. 
Then, by inheriting all the market-related terms (independent of $\omega$) from the planner's Lagrangian \eqref{eq:base_Lagrangian}, we obtain $L_m$ as 
\begin{small}
\beq
     L_m(q,\lambda,\eta;p)  : =   \alpha^T(q-p)       
		  - \lambda \cdot   \mathbf 1^T (q-d)  + \eta^{T} \left(  H^T(q-d) - F  \right) \ .
\eeq\end{small}%
It implies the explicit market problem formulation \eqref{eq:market_problem_general} as 
\begin{subequations}\label{eq:edp_quantity_bidding}%
	\begin{eqnarray}
		\label{eq:edp_quantity_bidding.a}
		U_m(q;p)  & := &     0   \\
		\label{eq:edp_quantity_bidding.b}
		G_m(q;p)\le 0  & : &  \eqref{eq:uedp.b}-\eqref{eq:uedp.d}
	\end{eqnarray}	
\end{subequations}

The market problem \eqref{eq:edp_quantity_bidding} is a feasibility problem that coordinates quantity bids of individual generation scheduling through pricing ($\pi = \lambda \cdot \mathbf{1} - H\eta - \omega$).
We need to further select \eqref{eq:market_gradient_general.a} for $( \lambda,\eta)$ and \eqref{eq:market_gradient_general.b} for $(q,\alpha)$ (to yield $\tilde L$ in \eqref{eq:aligned_participant_reduced_Lagrangian_quantity}).
The resulting formal aligned market dynamics are
\bseq\label{eq:gradient_play_market_quantity_bidding}
\begin{align}
\label{eq:gradient_play_market_quantity_bidding.a}
        	  q \ & \equiv \  p           \\
        	  \label{eq:gradient_play_market_quantity_bidding.b}
    	T^{\lambda} \dot \lambda \ & =\   -    \mathbf 1^T (q-d)   \\
    	 \label{eq:gradient_play_market_quantity_bidding.c}
	T^{\eta}  \dot \eta \ & =\  	 \left[ H^T(q-d)  - {F} \right]^+_{\eta}    
\end{align}
\eseq


Under the quantity bidding mechanism, the grid-market-participant loop, consisting of \eqref{eq:swingdynamics}, \eqref{eq:gradient_play_gen_quantity_bidding}, \eqref{eq:gradient_play_market_quantity_bidding}, equivalently implements the projected saddle flow \eqref{eq:grid-market-participant_loop} of the underlying reduced planner's Lagrangian $\hat L=\tilde L$ in \eqref{eq:aligned_participant_reduced_Lagrangian_quantity} that optimizes $\bar L$ over $(v_m, v_j,j\in\mathcal{N})$, i.e.,
\beq\label{eq:reduced_Lagrangian_quantity}
\hat L(z,\sigma) : = \min_{q}\max_{\alpha} \ \bar L(p,q,\tilde \theta, \alpha,\lambda,\eta,\nu ) \ ,
\eeq
with $z:= (p, \tilde \theta) \in\mathbb{R}^{|\mathcal{N}|+ |\mathcal{E}|}$ and $\sigma : = (\lambda,\eta,\nu)\in \mathbb{R}^{|\mathcal{N}| + 2|\mathcal{E}| + 1}$.
In this setting, the observability in Assumption~\ref{ass:observability} indeed holds as claimed below (proof in Appendix~\ref{apx:proof-of-prop_observability}).
\begin{proposition}\label{prop:observability_quantity}
Given $\hat L(z,\sigma)$ in \eqref{eq:reduced_Lagrangian_quantity}, if any trajectory $(z(t),\sigma(t))$ of the grid-market-participant loop \eqref{eq:grid-market-participant_loop} satisfies $\hat L(z^*,\sigma(t)) \equiv \hat L(z^*,\sigma^*)$ and $\hat L(z(t),\sigma^*) \equiv \hat L(z^*,\sigma^*)$, $\dot z, \dot \sigma \equiv 0$ holds.
\end{proposition}

\subsubsection{Price Bidding}

The price bidding mechanism allows individual generator participants to bid into the market the desired prices $\alpha$ of electricity at which they are willing to sell. Such a price bid is expected to implicitly reflect the marginal generation cost without revealing a generator's truthful cost function. The market targets an economic schedule of generation dispatch $q$ with the corresponding incentive compatible clearing prices $\pi$ based on these price bids.

Given the market dispatch $q_j$ and clearing price $\pi_j$ at bus $j\in\mathcal{N}$, generator $j$ is obliged to follow the designated generation dispatch but still can strive for profit maximization through its price bid:
\bseq\label{eq:gen_price_bidding_problem}
	\begin{eqnarray}
		\label{eq:gen_price_bidding_problem.a}
		U_j(p_j;q_j, \pi_j)  & := &      \pi_j q_j - J_j(p_j)    \\
		\label{eq:gen_price_bidding_problem.b}
		G_j(p_j;q_j, \pi_j) \le 0  & : & \eqref{eq:uedp.b}
	\end{eqnarray}	
\eseq
We interpret the price bid from the dual perspective and define a corresponding Lagrangian $L_j$:
\beq\label{eq:L_j_price}
L_j(p_j,\alpha_j;q_j, \pi_j) \ := \ \pi_j q_j - J_j(p_j) + \alpha_j (p_j - q_j) \ . 
\eeq
We adopt \eqref{eq:gen_gradient_general.a} for $\alpha$ and \eqref{eq:gen_gradient_general.b} for $p$ to characterize the bidding strategy of generator $j$ \cite{stegink2017frequency}:
\beq\label{eq:gradient_play_gen_price_bidding}
\tau_j^{\alpha} \dot \alpha_j  \ = \   q_j - ( \nabla J_j)^{-1}(\alpha_j) \ , \quad j\in\mathcal{N} \ , 
\eeq
where we have substituted $p_j = ( \nabla J_j)^{-1}(\alpha_j)$ from \eqref{eq:gen_gradient_general.b} for $p$.
The rational bidding behavior \eqref{eq:gradient_play_gen_price_bidding} reveals generator $j$'s effort to align its desired generation $(\nabla J_j)^{-1}(\alpha_j)$, conveyed through the price bid $\alpha_j$, with the given dispatch $q_j$. For example, an increase in $\alpha_j$ implies raised desired generation $(\nabla J_j)^{-1}(\alpha_j)$ (by convexity), and meanwhile signals the market to decrease the corresponding dispatch $q_j$, thus diminishing the gap in between.

It can be readily verified that \eqref{eq:gradient_play_gen_price_bidding} also satisfies Definition~\ref{def:alignment} of alignment with $\tilde L = \bar L$ in \eqref{eq:reduced_planner's_Lagrangian_bar_L}. Therefore, $v_m$ is $\varnothing$, while we have $v_j=p_j$, $j\in\mathcal{N}$, from \eqref{eq:gradient_play_gen_price_bidding}. Based on \eqref{eq:market_alignment}, we aim for $L_m$ such that $\nabla_u L_m = \nabla_u \min_p \bar L$ holds for $u=q,\lambda,\eta$. Note that now $p=q$ is accounted for by individual generators in \eqref{eq:gen_price_bidding_problem.b}. Then, inheriting all the market-related terms (independent of $p$) from the planner's Lagrangian \eqref{eq:base_Lagrangian}, including $\nu^T q$ and $\alpha^T q$ that capture interactions, leads to 
\begin{small}
\beq\label{eq:L_m_price_bidding}
      L_m(q,\lambda,\eta;\alpha,\omega) 
    : =     (\alpha+\omega)^T q   
		  - \lambda \cdot   \mathbf 1^T (q-d) + \eta^{T} \Big(  H^T(q-d) - F  \Big) 
\eeq
\end{small}%
with interchangeable $\omega \leftrightarrow \nu$.
It corresponds to the desired market problem as 
\begin{subequations}\label{eq:edp_price_bidding}%
	\begin{eqnarray}
		\label{eq:edp_price_bidding.a}
		U_m(q;\alpha,\omega)  & := &    (\alpha+\omega)^T q     \\
		\label{eq:edp_price_bidding.b}
		G_m(q;\alpha,\omega)\le 0  & : &  \eqref{eq:uedp.c}-\eqref{eq:uedp.d}
	\end{eqnarray}	
\end{subequations}

The market problem \eqref{eq:edp_price_bidding} uses the price bids and the frequencies, $\alpha+\omega$, as a proxy for the unit generation costs, and minimizes the corresponding aggregate generation cost -- a frequency-aware variant of economic dispatch.\footnote{The linear programming formulation \eqref{eq:edp_price_bidding} may not be well defined (finite) in general without generation capacity constraints. However, it is sufficient to derive control laws that make the closed loop steady-state optimal.}
Since $v_m$ is $\varnothing$, we will just implement \eqref{eq:market_gradient_general.a} for $(q,\lambda,\eta)$, leading to the formal aligned market dynamics:
\bseq\label{eq:gradient_play_market_price_bidding}
\begin{align}
\label{eq:gradient_play_market_price_bidding.a}
T^{q} \dot q \ & = \   \lambda \cdot \mathbf{1} - H \eta - \omega - \alpha  \\
        	  \label{eq:gradient_play_market_price_bidding.b}
    	T^{\lambda} \dot \lambda \ & =\   -    \mathbf 1^T (q-d)   \\
    	 \label{eq:gradient_play_market_price_bidding.c}
T^{\eta}	 \dot \eta \ & =\  	 \left[ H^T(q-d)  - {F} \right]^+_{\eta}    
\end{align}
\eseq



Under the price bidding mechanism, the grid-market-participant loop, consisting of \eqref{eq:swingdynamics}, \eqref{eq:gradient_play_gen_price_bidding}, \eqref{eq:gradient_play_market_price_bidding}, equivalently implements the projected saddle flow \eqref{eq:grid-market-participant_loop} of the underlying reduced planner's Lagrangian 
\beq\label{eq:reduced_Lagrangian_price}
\hat L(z,\sigma) : = \min_{p} \ \bar L(p,q,\tilde \theta, \alpha,\lambda,\eta,\nu ) \ ,
\eeq
based on \eqref{eq:reduced_Lagrangian_final}, with $z:= (q, \tilde \theta) \in\mathbb{R}^{|\mathcal{N}|+ |\mathcal{E}|}$ and $\sigma : = (\alpha,\lambda, \eta,\nu)\in \mathbb{R}^{2|\mathcal{N}| + 2|\mathcal{E}| + 1}$.
Here $p\equiv (\nabla J)^{-1} (\alpha)$ is enforced with the column vector-valued inverse function $( \nabla J)^{-1}(\cdot):= ( (\nabla J_j)^{-1}(\cdot),j\in\mathcal{N}) :\mathbb{R}^{|\mathcal{N}|} \rightarrow \mathbb{R}^{|\mathcal{N}|}$ representing the element-wise inverse of gradients.
The observability in Assumption~\ref{ass:observability} likewise holds here as formally stated below (proof analogous to that of Proposition~\ref{prop:observability_quantity}).
\begin{proposition}\label{prop:observability_price}
Given $\hat L(z,\sigma)$ in \eqref{eq:reduced_Lagrangian_price}, if any trajectory $(z(t),\sigma(t))$ of the grid-market-participant loop \eqref{eq:grid-market-participant_loop} satisfies $\hat L(z^*,\sigma(t)) \equiv \hat L(z^*,\sigma^*)$ and $\hat L(z(t),\sigma^*) \equiv \hat L(z^*,\sigma^*)$, $\dot z, \dot \sigma \equiv 0$ holds.
\end{proposition}

\section{Misaligned Market Dynamics} \label{sec:misaligned_MD}

In general, individual participants are not obliged to conform with any behavior pattern aligned with the market. 
If they do not bid in an aligned manner, the role of market dynamics has to be reevaluated.
In this section, we propose an exemplar of rational price bidding strategy for price-taking participants that is misaligned. 
To maintain the desirable properties of the market dynamics, a modification on the market control laws is necessary to accommodate the misalignment.

\subsection{Misaligned Price Bidding}

We still consider the dynamic price bidding mechanism where each generator $j$ interacts with the market through its price bid $\alpha_j$.
From a generator's perspective, given the market dispatch $q_j$ and clearing price $\pi_j$, one  alternative to fully exploit the market information is to check whether the pair $(q_j,\pi_j)$ satisfies incentive compatibility, i.e., whether its marginal generation cost matches the clearing price, $\nabla J_j(q_j)=\pi_j$.
Note that 
\beq
( \nabla J_j)^{-1}(\pi_j) = \arg\max_{p_j} \ \pi_j p_j - J_j(p_j)
\eeq
is indeed the individual desired output that maximizes generator $j$'s profit given the clearing price $\pi_j$. 
If this output is matched by the market dispatch $q_j$, the generator should be satisfied with its current clearing price and dispatch. 
Otherwise, any discrepancy  between the individual desired output and the market dispatch would drive generator $j$ to strive for compatible incentives. 

This idea inspires a new formulation of the individual price bidding problem:
\bseq\label{eq:gen_price_bidding_problem_alternative}
\bq
	\label{eq:gen_price_bidding_problem_alternative.a}
		U_j(p_j;q_j, \pi_j)  & := &     0   \\
		\label{eq:gen_price_bidding_problem_alternative.b}
		G_j(p_j;q_j, \pi_j) \le 0  & : & \left\{
		\begin{aligned}
		    & \eqref{eq:uedp.b} \\
		    &  p_j =( \nabla J_j)^{-1}(\pi_j) 
		\end{aligned}\right.
\eq
\eseq
where \eqref{eq:gen_price_bidding_problem_alternative.b} still enforces the market dispatch $q_j$ to be strictly followed while generator $j$ requires it to be incentive compatible with the clearing price $\pi_j$.
This is a feasibility problem and we still interpret the price bidding from the dual perspective. We define $L_j$ to be a partial Lagrangian of \eqref{eq:gen_price_bidding_problem_alternative} that only relaxes \eqref{eq:uedp.b}:
\beq
L_j(\alpha_j;q_j,\pi_j) \ := \ \alpha_j\left(( \nabla J_j)^{-1}(\pi_j)  - q_j \right) \ ,
\eeq
where we have plugged in $p_j =( \nabla J_j)^{-1}(\pi_j) $ to indicate the individual desired output. The dynamic gradient play \eqref{eq:gen_gradient_general.a} on $\alpha$ defines an alternative price bidding strategy for generator $j$:
\beq\label{eq:gradient_play_gen_price_bidding_alternative}
\tau^{\alpha}_j \dot \alpha_j \ = \   q_j -( \nabla J_j)^{-1}(\pi_j)   \ , \quad j\in\mathcal{N} \ .
\eeq

Compared with the previous aligned bidding behavior \eqref{eq:gradient_play_gen_price_bidding}, the clearing price information is exploited here instead of each local bid. To some extent, the current strategy \eqref{eq:gradient_play_gen_price_bidding_alternative} is even more straightforward and compelling for rational individual generators since it directly reflects their economic incentives. For instance, if a generator is dispatched more than its desired generation output, it raises its price bid to indicate more costly production, in anticipation of a reduced dispatch or a lifted clearing price, at which it is willing to output more. 
The price bid $\alpha_j$ will remain fixed only when the pair of the market dispatch $q_j$ and clearing price $\pi_j$ satisfies $\nabla J_j(q_j) = \pi_j$ to be incentive compatible. 
Note that a generator bidding according to \eqref{eq:gradient_play_gen_price_bidding_alternative} 
is still a price taker since it just responds to the given dispatch and clearing price.

Suppose the market still maintains the market control laws \eqref{eq:gradient_play_market_price_bidding}, attained from the market problem \eqref{eq:edp_price_bidding} under the price bidding mechanism. 
The alignment condition does not hold here for such a grid-market-participant loop since there does not exist any reduced planner's Lagrangian that satisfies Definition~\ref{def:alignment}. In fact, the insertion of $p = (\nabla J)^{-1}(\pi)$ into the planner's Lagrangian \eqref{eq:base_Lagrangian} could deprive it of the concavity in the dual variables $(\alpha,\lambda,\eta, \nu)$.
The next illustrative example further suggests that such misalignment can lead to system instability.

\noindent
\textbf{Single-bus example}

\textit{
We adopt a quadratic form for $J_j(\cdot)$ as
\beq\label{eq:quadratic_gen_function}
    J_j(q_j) :=    \frac{c_j}{2} q_j^2  + \bar c_j q_j \ ,
\eeq
which is parametrized by constants $c_j >0$ and $\bar c_j $.
Consider the following illustrative single-bus example. 
We can ignore the state variables of $\theta, \eta$ as a result and further drop all the subscripts.
Suppose there is a step change in power demand $d$. The coupled system is given by
\bseq
\begin{align}
    \dot \omega & \  = \  q -d - D\omega  \\
     \dot \alpha & \ = \ q - c^{-1} (\lambda - \omega - \bar c) \\
    \dot q & \ = \ \lambda - \omega - \alpha \\
    \dot \lambda & \ = \  - (q -d) 
\end{align}
\eseq
where the inertia and all the updating step sizes are set to $1$. For computational convenience, we further look at the case with $D=1$ and $c=1$, and write the coupled system in a compact form as
\beq\label{eq:toy_example}
\begin{bmatrix}
\dot \omega \\ \dot \alpha \\ \dot q \\ \dot \lambda
\end{bmatrix} = 
\underbrace{
\begin{bmatrix}
-1 & 0 & 1 & 0  \\
1 & 0 & 1 & -1 \\
-1 & -1 & 0 & 1 \\
0 & 0 & -1 & 0
\end{bmatrix}}_{=:A}
\begin{bmatrix}
\omega \\  \alpha \\  q \\  \lambda
\end{bmatrix} +
\underbrace{
\begin{bmatrix}
-d \\ \bar c \\  0 \\ d
\end{bmatrix}}_{\textrm{constant}} \  .
\eeq
It can be checked that the matrix $A$ has a pair of complex-conjugate eigenvalues $0.16 \pm \mathbf{i} 1.75 $ with positive real parts. Thus, the system \eqref{eq:toy_example} is not stable, as also illustrated in Fig.~\ref{fig:toy}.}

\begin{figure}[t!]
\centering
\includegraphics[width=\columnwidth]{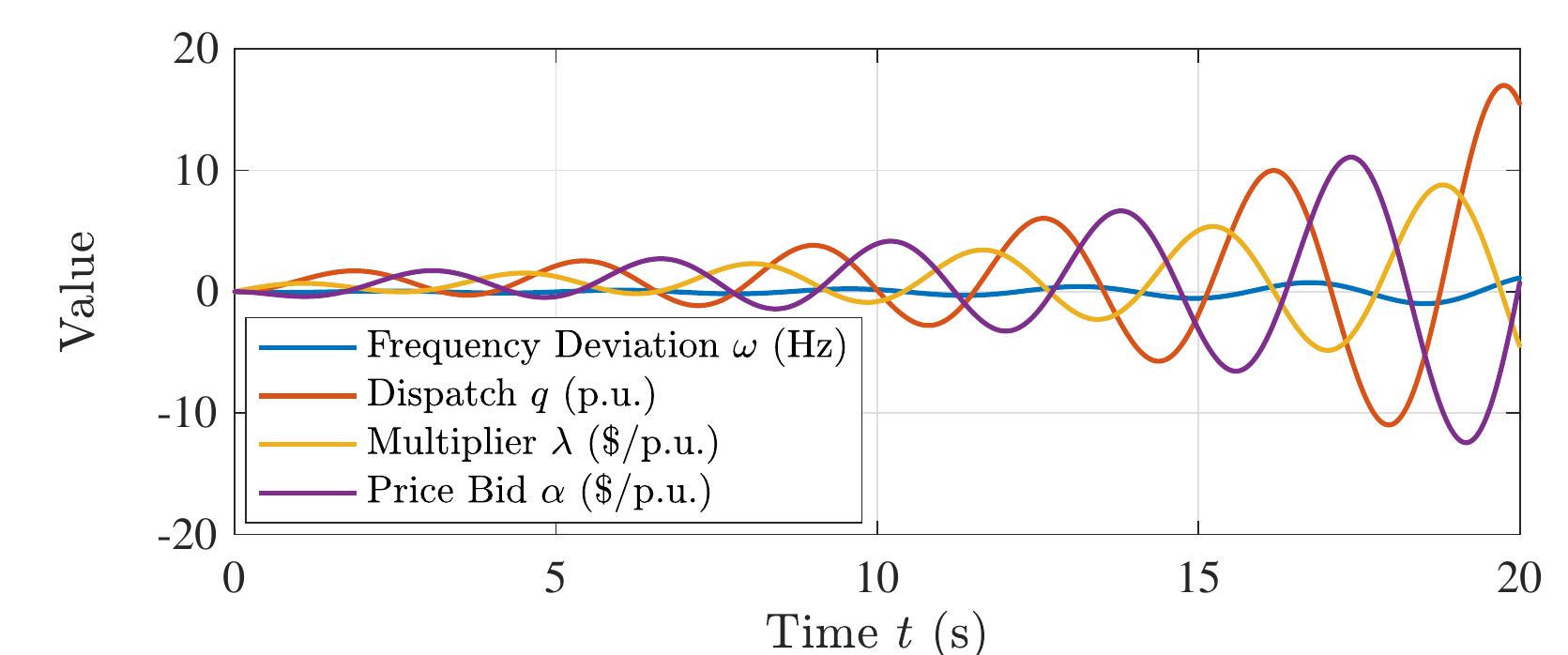}
\caption{Simulation run of illustrative example \eqref{eq:toy_example}.}
\label{fig:toy}
\end{figure}

The example indicates that individual participants being rational price-takers is not sufficient
to guarantee alignment between participants and the grid planner. Further, such misalignment can render the design of saddle-flow based market dynamics \eqref{eq:gradient_play_market_price_bidding} closed-loop unstable.
The source of the misalignment in this case can be understood as having participants' bidding dynamics based on saddle flows of different reduced Lagrangians.

More precisely, for the closed-loop system \eqref{eq:swingdynamics}, \eqref{eq:gradient_play_market_price_bidding}, \eqref{eq:gradient_play_gen_price_bidding_alternative}, one can identify that individual participants choose the desired generation outputs $p^{\ddagger}$ from the following reduced planner's Lagrangian:
\begin{align*}
    p^{\ddagger}:= & \arg\min_{p} \left\{ \max_{\alpha}  \min_{q}  \  \bar L(p,q , \tilde \theta, \alpha, \lambda, \eta,\nu) \right\}
     \\
    =& (\nabla J)^{-1}(\lambda\cdot \mathbf 1 - H\eta -\nu)  \ .
\end{align*}
However, all the remaining decisions are taken from another reduced planner's Lagrangian (with $p=p^{\ddagger}$):
\beq\label{eq:saddle_flow_variant}
\begin{aligned}
    T^{z} \dot z & = - \nabla_{z} \bar L(p,q , \tilde \theta, \alpha, \lambda, \eta,\nu) \Big\vert_{p=p^{\ddagger}}\\ 
    T^{\sigma} \dot \sigma & = \left[  \nabla_{\sigma} \bar L(p,q , \tilde \theta, \alpha, \lambda, \eta,\nu) \Big\vert_{p=p^{\ddagger}}\right]^+_{\eta}
\end{aligned}
\eeq
with $z=(q,\tilde \theta)$ and $\sigma= (\lambda,\nu,\alpha, \eta)$. This variant \eqref{eq:saddle_flow_variant} of saddle flow dynamics matches exactly the current closed-loop system
\eqref{eq:swingdynamics}, \eqref{eq:gradient_play_market_price_bidding}, \eqref{eq:gradient_play_gen_price_bidding_alternative}. Obviously, participants' bidding behavior is not aligned since they adopt two different reduced Lagrangians to update their variables $p$ and $\alpha$.
We can also expect the malfunction of \eqref{eq:saddle_flow_variant}, due to the improper enforcement of a minimizer from a different reduced Lagrangian.


\subsection{Market Modification}

In this subsection we propose a solution that modifies the market control laws \eqref{eq:gradient_play_market_price_bidding} to accommodate such misaligned bidding behavior \eqref{eq:gradient_play_gen_price_bidding_alternative}. To simplify the analysis and presentation, we still use the quadratic generation cost in \eqref{eq:quadratic_gen_function}.
We introduce an auxiliary primal variable $\hat q  \in\mathbb{R}^{\vert\mathcal{N} \vert}$ and reformulate the market problem with an extra regularization term in the objective function to minimize:
\begin{small}
\begin{subequations}\label{eq:edp_reg_price_bidding}
	\begin{eqnarray}
		\label{eq:edp_reg_price_bidding.a}
		U_m(q,\hat q;\alpha,\omega) \!\!\!\! & := &  \!\!\!\!   (\alpha+\omega)^T q  
	  + \frac{\rho}{2}  \big\Vert q-\hat q \big \Vert^2     \\
		\label{eq:edp_reg_price_bidding.b}
		G_m(q,\hat q;\alpha,\omega)\le 0  \!\!\!\! & : & \!\!\!\!  \eqref{eq:uedp.c}-\eqref{eq:uedp.d}
	\end{eqnarray}	
\end{subequations}\end{small}%
where $\rho>0$ is a constant regularization coefficient.
Due to the positive regularization, the minimum of \eqref{eq:edp_reg_price_bidding} is lower bounded by that of \eqref{eq:edp_price_bidding} and the bound is tight only when $q^*=\hat q^*$ holds. Therefore, $(q^*,\hat q^*= q^*,  \lambda^*,\eta^*)$ is optimal w.r.t. \eqref{eq:edp_reg_price_bidding} if and only if $(q^*, \lambda^*,\eta^*)$ is optimal w.r.t. \eqref{eq:edp_price_bidding}. 
In this sense, we refer to $\hat q$ as \emph{virtual dispatch} since it is consistent with real dispatch $q$ at optimality.

We still follow the general idea of gradient-based market control laws to define the Lagrangian $L_m$ of \eqref{eq:edp_reg_price_bidding}:
\begin{small}
\beq
\begin{aligned}
     L_m(q,\hat q, \lambda,\eta;\alpha,\omega)  
    := & \   (\alpha+\omega)^T q +     \frac{\rho}{2}  \big\Vert q-\hat q \big \Vert^2 \\
		&  - \lambda \cdot   \mathbf 1^T (q-d)   
		    + \eta^{T} \left(  H^T(q-d) - F  \right)   \  ,
\end{aligned}
\eeq\end{small}%
and select \eqref{eq:market_gradient_general.a} for $(\hat q, \lambda, \eta)$ and \eqref{eq:market_gradient_general.b} for $q$, which yields the modified market dynamics
\begin{small}
\bseq\label{eq:regularized_market_graident_play}
\begin{align}
q \ & \equiv \   \frac{1}{\rho} ( \lambda\cdot \mathbf{1} - H\eta -\omega - \alpha)  +\hat q    
\label{eq:regularized_market_graident_play.a}\\ 
T^{\hat q } \dot {\hat q} \ &= \     \lambda\cdot \mathbf{1} - H\eta -\omega - \alpha   \label{eq:regularized_market_graident_play.b}\\
T^{\lambda} \dot \lambda \ &= \  -    \mathbf 1^T \left(   \frac{1}{\rho}  (\lambda \cdot \mathbf{1} - H\eta -  \omega  - \alpha )   + \hat q -   d   \right)  \label{eq:regularized_market_graident_play.c}\\
T^{\eta} \dot \eta \ &= \ \left[ H^T\left(    \frac{1}{\rho}  (\lambda \cdot \mathbf{1} - H\eta -  \omega  -\alpha ) +\hat q -  d \right)- {F}  \right]^+_{\eta}   \label{eq:regularized_market_graident_play.d} 
\end{align}
\eseq\end{small}%
The virtual dispatch $\hat q$ is internal to the market, while the real dispatch $q$ is released to individual participants for implementation along with prices $\pi=\lambda \cdot \mathbf{1} - H\eta -  \omega$.

\subsection{Equilibrium Analysis and Asymptotic Stability}
In this subsection we examine the interaction among the physical power network dynamics \eqref{eq:swingdynamics}, the modified market control laws \eqref{eq:regularized_market_graident_play} and the rational bidding behavior \eqref{eq:gradient_play_gen_price_bidding_alternative} of individual participants. 
We will formally characterize the steady state and stability of this new closed-loop system.

We first define $z:=(\hat q, \tilde \theta)\in \mathbb{R}^{|\mathcal{N}| + |\mathcal{E}| }$ and $\sigma : = (\lambda,\omega,\alpha,\eta) \in\mathbb{R}^{2|\mathcal{N}|+ 2|\mathcal{E}|+1}$, and note that $q$ in \eqref{eq:regularized_market_graident_play.a} is the reduced variable.
We retain the sets of initial points and equilibrium points defined in \eqref{eq:initial_set} and \eqref{eq:equilibrium_set}, respectively. The corresponding equilibrium set is then explicitly characterized as follows (proof in Appendix~\ref{apx:proof_of_teo_eqm_reg}).
\begin{theorem}\label{teo:eqm_reg}
	For the grid-market-participant loop \eqref{eq:swingdynamics}, \eqref{eq:gradient_play_gen_price_bidding_alternative}, \eqref{eq:regularized_market_graident_play} that starts from any initial point in $\mathbb{I}$, a point $(z^*,\sigma^*)$ is an equilibrium if and only if $(z^*,\sigma^*)$ corresponds to an optimal primal-dual solution to the planner's problem \eqref{eq:uedp}.
\end{theorem}


Having secured the desirable steady state, we next show that given the initial condition of $\mathbb{I}$, the grid-market-participant loop \eqref{eq:swingdynamics}, \eqref{eq:gradient_play_gen_price_bidding_alternative}, \eqref{eq:regularized_market_graident_play} indeed converges to one equilibrium point in $\mathbb{E}$ with proper choice of the coefficient $\rho$, as summarized below (proof in Appendix~\ref{apx:proof_of_teo_stability_reg}).

\begin{theorem}\label{teo:stability_reg}
The equilibrium set $\mathbb{E}$ is globally asymptotically stable on $\mathbb{I}$, given $\rho\in \left(0,\inf_{j\in\mathcal{N}}   4c_j \right)$. In particular, starting from any initial point in $\mathbb{I}$, a trajectory $(z(t),\sigma(t))$ of the grid-market-participant loop \eqref{eq:swingdynamics}, \eqref{eq:gradient_play_gen_price_bidding_alternative}, \eqref{eq:regularized_market_graident_play} remains bounded for $t\ge 0$ and converges to $(z^*,\sigma^*)$ with $t\rightarrow\infty$, where $(z^*,\sigma^*)$ is one specific equilibrium point in $\mathbb{E}$.
\end{theorem}
\begin{remark}
Recall $c_j>0$, $j\in\mathcal{N}$ is the given quadratic coefficient in each generation cost function \eqref{eq:quadratic_gen_function}. Therefore, we can always pick a sufficiently small $\rho$ that satisfies the condition such that the grid-market-participant loop \eqref{eq:swingdynamics}, \eqref{eq:gradient_play_gen_price_bidding_alternative}, \eqref{eq:regularized_market_graident_play} asymptotically converges to a target equilibrium point.
\end{remark}

Theorem~\ref{teo:stability_reg} suggests that the proposed market modification with regularization \eqref{eq:edp_reg_price_bidding} is able to accommodate the misalignment from the way \eqref{eq:gradient_play_gen_price_bidding_alternative} each participant bids, such that the modified market dynamics and the misaligned individual bidding behavior can still function as a feedback controller on the power network dynamics to realize all the primary, secondary and tertiary control with compatible participants' incentives.
The required modification, however, highlights the necessity of \emph{robust} market control laws that can better accommodate diverse participants' behavior and potential misalignment. 


\subsection{Underlying Rationale: Implicit Regularization}
In this subsection we provide more intuitions about the modified market dynamics design \eqref{eq:regularized_market_graident_play}.
We first point out that the minimizer of the real dispatch $q$ in \eqref{eq:regularized_market_graident_play.a} is equivalent to the proximal operator (as a function of $\hat q$) associated with the original $L_m$ (in the variable $q$) in \eqref{eq:L_m_price_bidding}, 
and is commonly used in dealing with non-differentiable optimization problems \cite{dhingra2018proximal}. 
In light of the linear programming formulation of the original market problem \eqref{eq:edp_price_bidding}, we can expect this proximal gradient based market control laws \eqref{eq:regularized_market_graident_play} to secure convergence under milder conditions 
from an optimization perspective.

We note, however, from a design perspective, the proposed market modification does not recover alignment to connect with standard saddle flows. Indeed, the current grid-market-participant loop \eqref{eq:swingdynamics}, \eqref{eq:gradient_play_gen_price_bidding_alternative}, \eqref{eq:regularized_market_graident_play} still corresponds to the following improper form:
\begin{small}
\beq\label{eq:saddle_flow_variant_reg}
\begin{aligned}
    T^{z} \dot z & = - \nabla_{z}\left\{ \min_{q} \bar L(p,q , \tilde \theta, \alpha, \lambda, \eta,\nu) + \frac{\rho}{2}\Vert q-\hat q\Vert^2 \right\}\Big\vert_{p=p^{\ddagger}}\\ 
    T^{\sigma} \dot \sigma & = \left[  \nabla_{\sigma}  \left\{ \min_{q} \bar L(p,q , \tilde \theta, \alpha, \lambda, \eta,\nu)+ \frac{\rho}{2}\Vert q-\hat q\Vert^2 \right\}\Big\vert_{p=p^{\ddagger}}\right]^+_{\eta}
\end{aligned}
\eeq\end{small}%
with $z=(\hat q, \tilde \theta)$ here.
Compared with \eqref{eq:saddle_flow_variant} ($\hat q$ in place of $q$), the only difference between the target convex-concave functions (with respective minimizers plugged in) 
is an extra implicit regularization term 
\beq\label{eq:implicit_reg}
-\frac{1}{2\rho} \Vert \lambda\cdot \mathbf{1}- H\eta-\nu-\alpha \Vert^2  \ .
\eeq

Note that essentially the misaligned price bidding behavior \eqref{eq:gradient_play_gen_price_bidding_alternative} deviates from the aligned one \eqref{eq:gradient_play_gen_price_bidding} by using the clearing price information $\pi=\lambda\cdot \mathbf{1}- H\eta-\omega$ instead of the local bids $\alpha$. 
Thus, one can understand the role of the regularization term \eqref{eq:implicit_reg} as penalizing the mismatch between the clearing prices $\pi$ and the local bids $\alpha$. The smaller $\rho$ is, the stronger this regularization impact will be and the closer the clearing prices $\pi$ and the local bids $\alpha$ will be tied to each other.
Theorem~\ref{teo:stability_reg} provides a threshold for $\rho$ under which the two quantities are close enough such that the misalignment can be accommodated despite the use of the clearing price information.

\section{Numerical Results}\label{sec:simulation}

We test the proposed aligned market dynamics based on quantity and price bidding, as well as the modified market dynamics for misaligned price bidding on the IEEE $39$-bus system to illustrate their interplay with grid dynamics and respective participants' behavior. 
Despite basing our analysis on a linear approximation of the physical swing dynamics, we adopt a high-fidelity model for the numerical tests, including nonlinear power flows and voltage dynamics. 
All of the $10$ generators, located at buses $30$-$39$, are taken as market participants. 
We randomly select three lines $4$, $19$ and $26$ and endow them with relatively small transmission capacity $\pm 3$ p.u. such that the line thermal limits will be binding.
$1$ p.u. ($100$ MW) of step load increase is imposed at bus $30$ at time $0$.  
The $150$-second simulation runs of transient dynamics in response to this instant power imbalance for all the three grid-market-participant loops are presented in Fig.~\ref{fig:compare-bid}.

As we can observe from Fig.~\ref{fig:compare-bid}, 
all the three market dynamics are able to drive the respective systems to steady state within $150$ seconds. At equilibrium, they all restore the frequency to its nominal value, respect line thermal limits, and achieve the optimal generation dispatch, which is the unique solution to the planner's problem \eqref{eq:uedp}, with an identical set of incentive compatible clearing prices.
Fig.~\ref{fig:compare-bid.dispatch} and Fig.~\ref{fig:compare-bid.gen} indicate the convergence of the market dispatch $q$ and individual generators' desired output $p$ to the same value, which further suggests that their incentives are aligned.
In the price bidding setting, the generators' bids always converge to their local prices, as displayed in Figure~\ref{fig:compare-bid.pricebid}, despite different market information used in the bidding strategies. 

In addition, Fig.~\ref{fig:rho} shows a representative profile of frequency deviation at bus $34$ with the modified market dynamics for misaligned price bidding, as we vary the regularization coefficient $\rho$ subject to Theorem~\ref{teo:stability_reg}. 
A smaller $\rho$, e.g., $0.1$, enhances the role of the regularization in \eqref{eq:implicit_reg} and amplifies oscillations, while a larger $\rho$, e.g., $10$, damps the oscillations and leads to smoother convergence.

\begin{figure*}[t!]
\centering
\subfigure[Frequency deviation]
{\includegraphics[width=250pt, height=91pt]{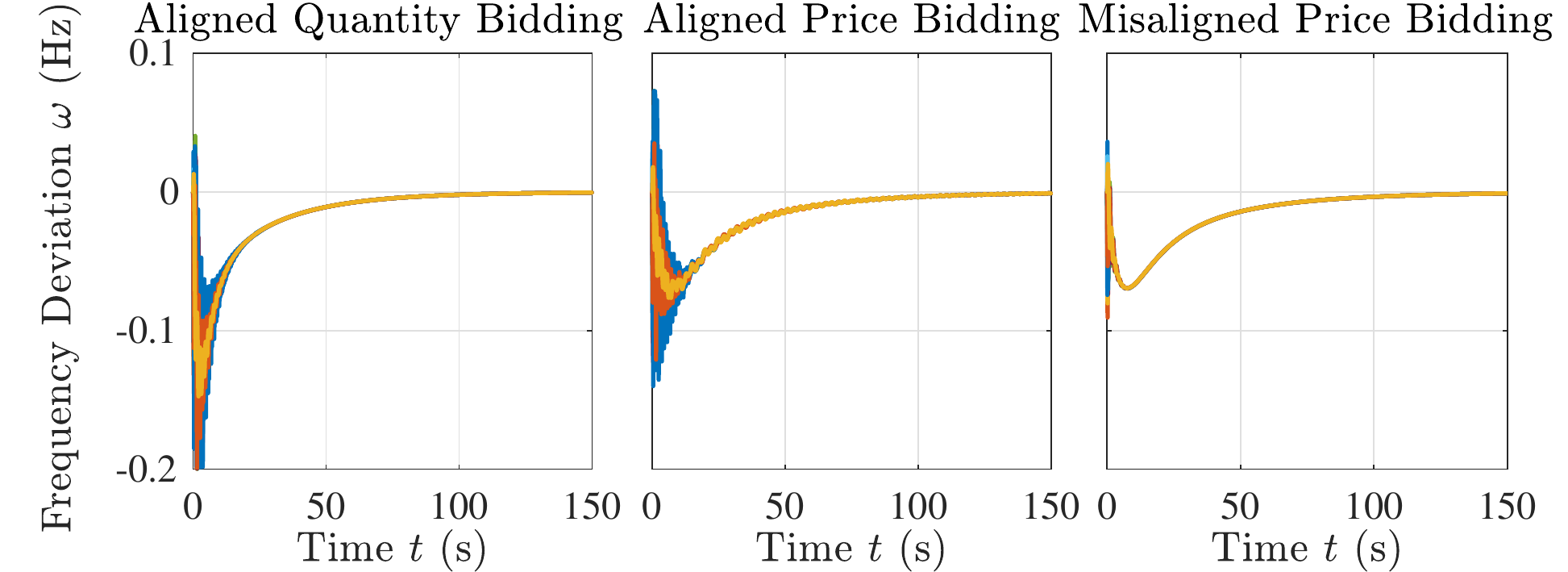}\label{fig:compare-bid.freq}}
\hfill
\subfigure[Selected line flows (thermal limits indicated by dashed lines)]
{\includegraphics[width=250pt, height=91pt]{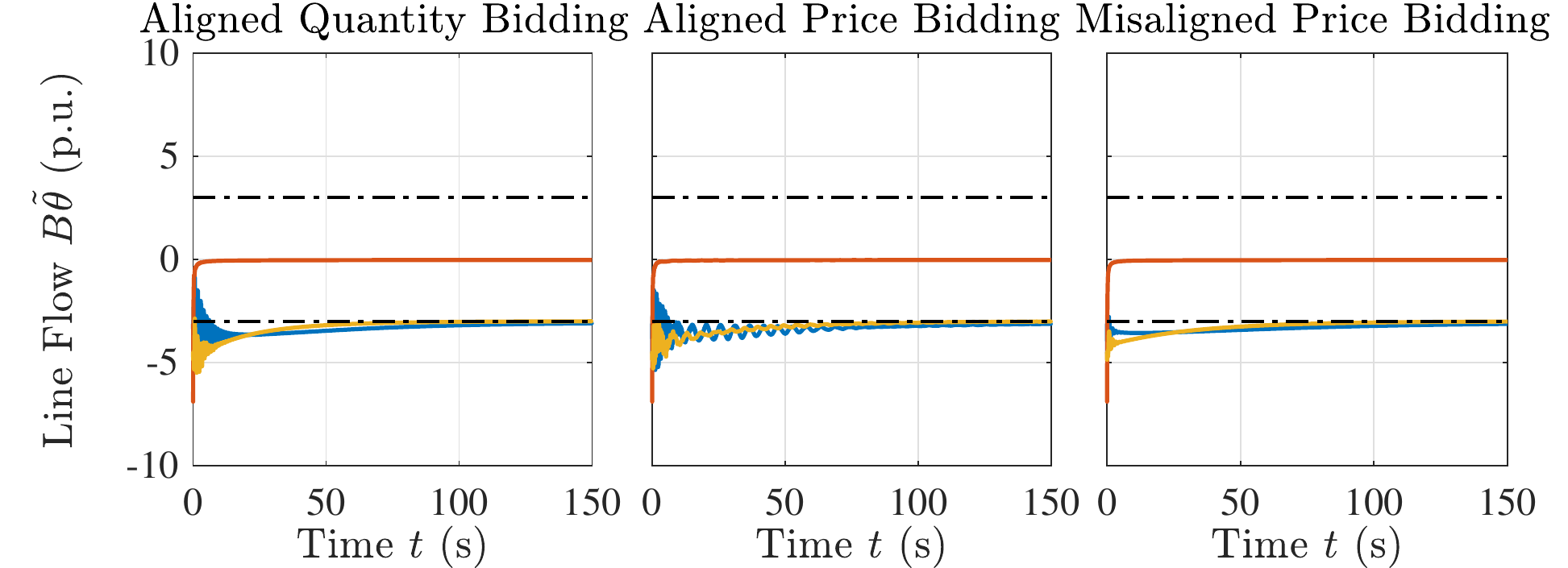}\label{fig:compare-bid.flow}}
\hfill
\subfigure[Generation dispatch]
{\includegraphics[width=250pt, height=91pt]{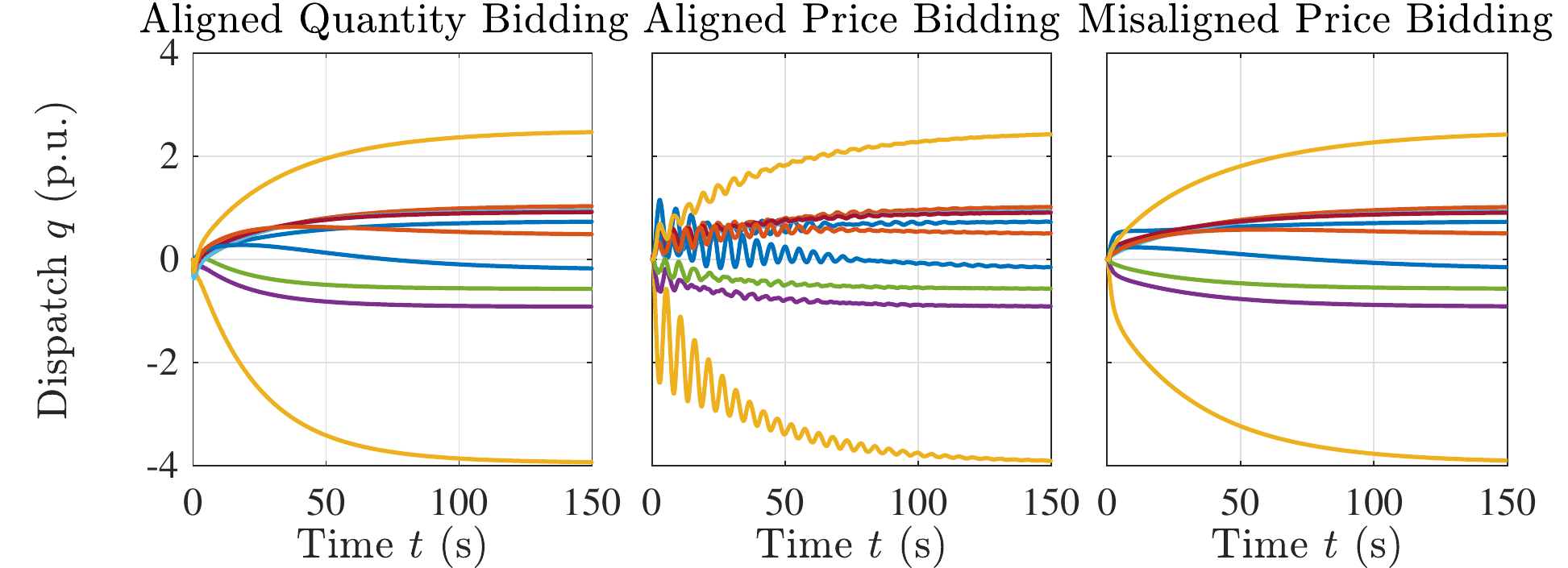}\label{fig:compare-bid.dispatch}}
\hfill
\subfigure[Individual desired generation]
{\includegraphics[width=250pt, height=91pt]{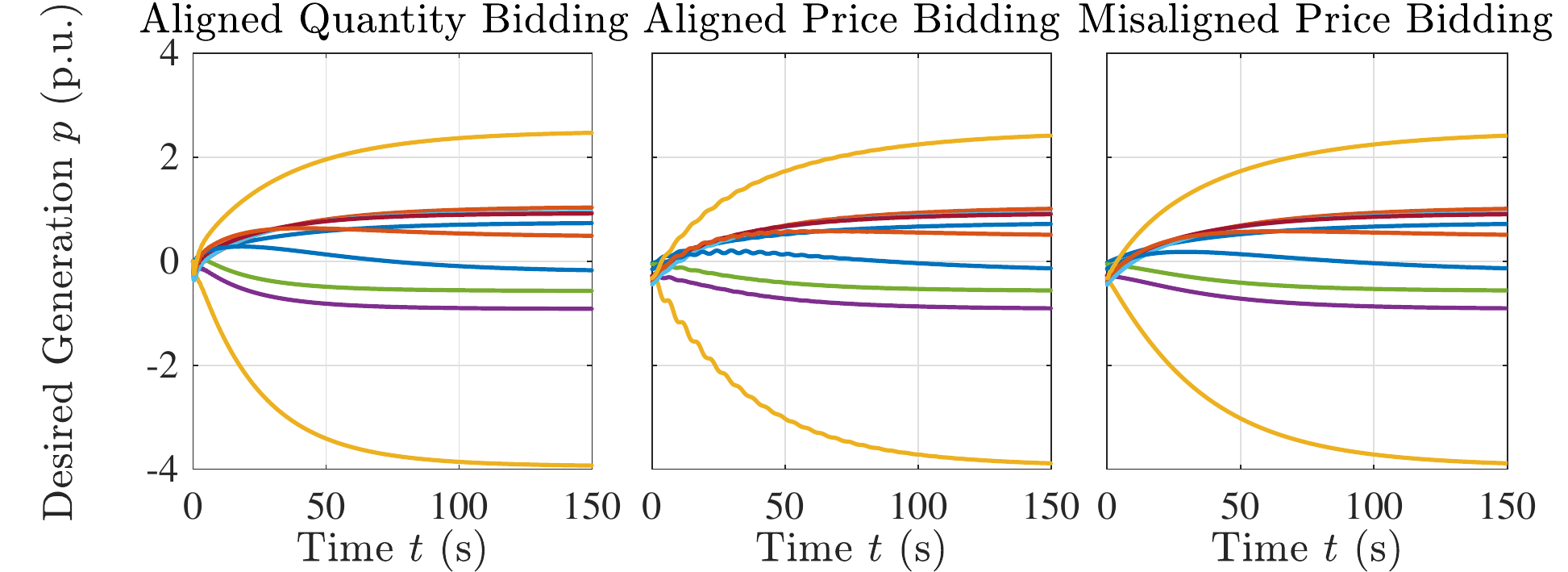}\label{fig:compare-bid.gen}}
\hfill
\subfigure[Clearing prices]
{\includegraphics[width=250pt, height=91pt]{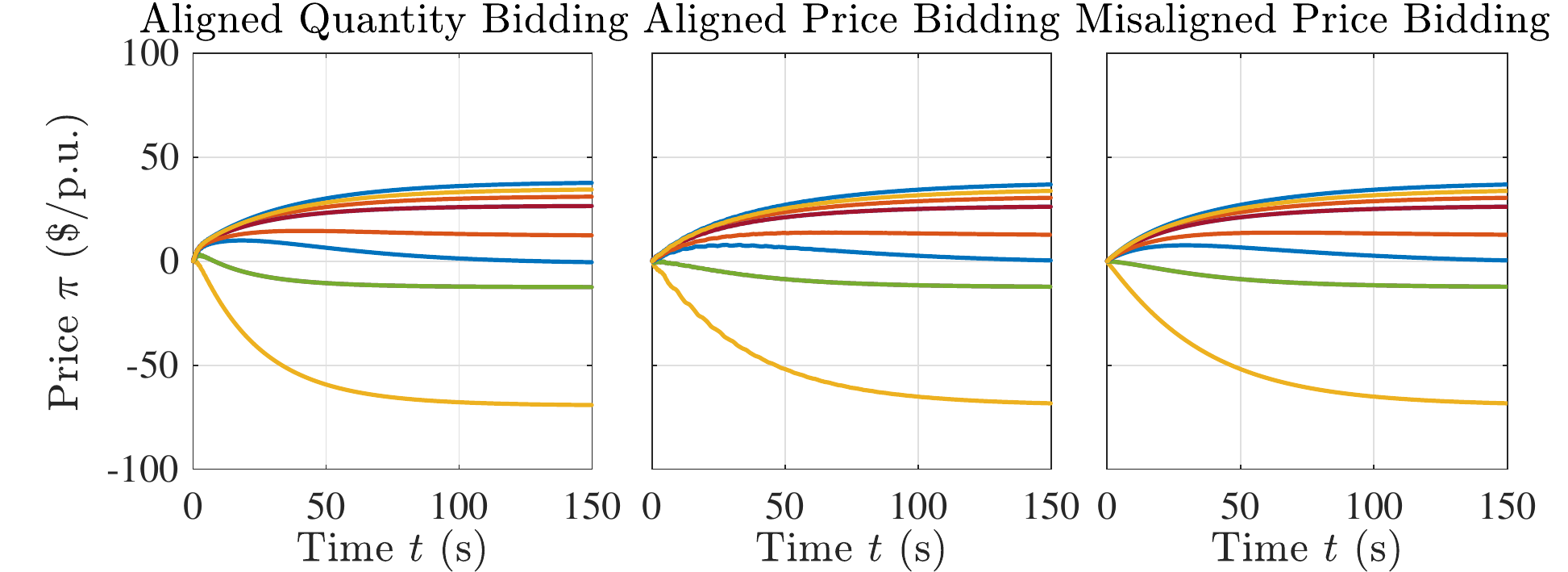}\label{fig:compare-bid.price}}
\hfill
\subfigure[Price bids]
{\includegraphics[width=250pt, height=91pt]{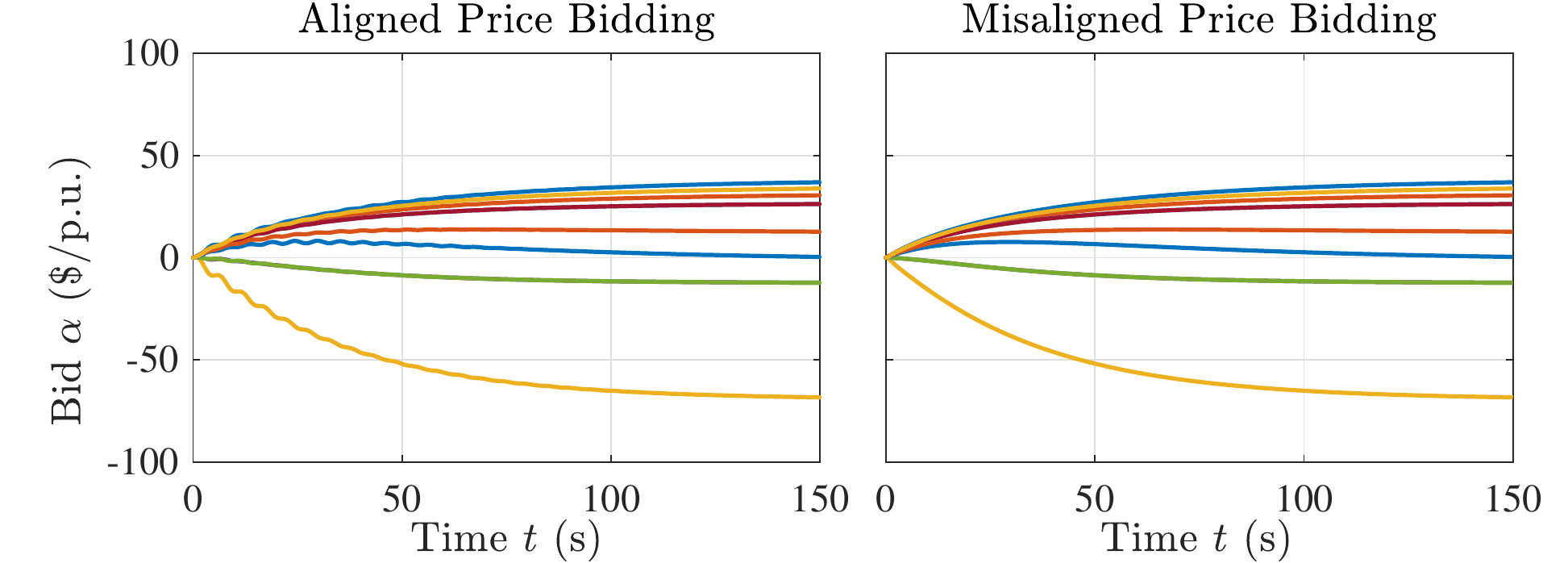}\label{fig:compare-bid.pricebid}}
\caption{Comparison across different market dynamics on instant power imbalance.}
\label{fig:compare-bid}
\end{figure*}

\begin{figure}[t!]
\centering
\includegraphics[width=235pt, height=100pt]{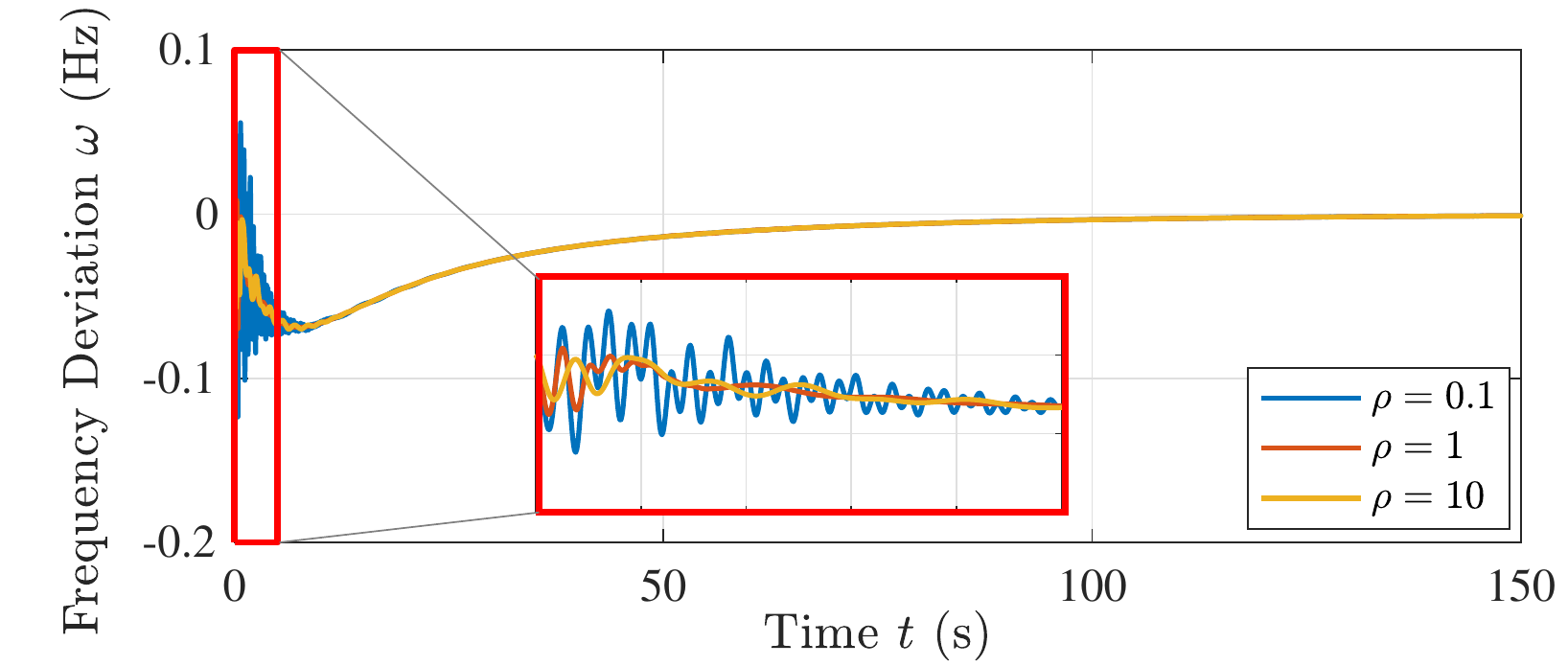}
\caption{Modified market dynamics for misaligned price bidding: impact of $\rho$ on frequency deviation at bus 34.}
\label{fig:rho}
\end{figure}




\section{Conclusion} \label{sec:conclusion}

This paper studies the interaction among grid dynamics, market dynamics and bidding dynamics of individual market participants.
We first develop a principled framework for the design and analysis of aligned market dynamics, conditioned on a family of alignment characterization for individual bidding behavior. 
The alignment connects the associated grid-market-participant loop to a saddle flow whose min-max saddle points, i.e., equilibrium points, optimally solve the target planner's problem to realize the primary, secondary, and tertiary frequency control with compatible participants' incentives.
We show that the framework is general for the bidding mechanisms based on quantity and price, and, under mild conditions, the asymptotic convergence of the closed-loop system to an equilibrium can be guaranteed.
Two specific examples of aligned market dynamics are demonstrated. 
We further study an exemplar of misaligned bidding behavior that can lead to system instability in the absence of the alignment property. 
A solution of modified market control laws is proposed to accommodate such misalignment, highlighting the necessity of market robustness against diverse individual bidding behavior.
Numerical simulations on the IEEE 39-bus system validate our market dynamics designs in terms of steady-state equilibrium and global asymptotic stability.

\appendix

\subsection{Proof of Theorem \ref{teo:planner}} \label{apx:proof_of_teo_planner}

Before we prove the sufficient and necessary conditions, we first show that any optimal primal-dual solution $(p^*,q^*,\omega^*,\tilde \theta^*,\alpha^*,\lambda^*,\eta^*,\nu^*)$ to the planner's problem \eqref{eq:uedp} satisfies $\omega^*  = 0 $.
Since \eqref{eq:uedp.b}-\eqref{eq:uedp.e} are affine, the KKT conditions are both sufficient and necessary to characterize the optimality of \eqref{eq:uedp}:

\bseq\label{eq:UEDP_KKT}
\noindent
\emph{Stationarity}
\begin{align}\label{eq:UEDP_stationarity.a}
& \nabla J(p^*) - \alpha^*   = 0 \\
\label{eq:UEDP_stationarity.b}
& \alpha^* - \lambda^* \cdot \mathbf{1} +  H \eta^*  +  \nu^*   = 0  \\
\label{eq:UEDP_stationarity.c}
& D(\omega^*-\nu^* ) = 0 \\
\label{eq:UEDP_stationarity.d}
& BC^T \nu^* = 0 
\end{align}

\noindent
\emph{Primal feasibility}
\begin{align}\label{eq:UEDP_primalfeas.e}
& q^* = p^* \\
\label{eq:UEDP_primalfeas.f}
& \mathbf{1}^T(q^*-d ) =0 \\
\label{eq:UEDP_primalfeas.g}
& H^T(q^*-d ) \le F  \\
\label{eq:UEDP_primalfeas.h}
& q^*-d  - D \omega^* - CB \tilde \theta^* = 0 
\end{align}

\noindent
\emph{Dual feasibility}
\begin{align}\label{eq:UEDP_dualfeas.i}
\eta^* \ge 0
\end{align}

\noindent
\emph{Complementary slackness}
\begin{align}\label{eq:UEDP_complementary.j}
\diag(\eta^*) \left( H^T(q^*-d) -F \right)  =0
\end{align}
\eseq

\eqref{eq:UEDP_stationarity.d} implies $C^T \nu^* = 0$ due to $B \succ 0 $. Since $C$ is the incidence matrix of the connected graph $(\mathcal{N},\mathcal{E})$, $C^T \nu^* = 0 $ basically means $\nu_j^* = \zeta$, $\forall j\in\mathcal{N}$, where $\zeta$ is a constant.
\eqref{eq:UEDP_stationarity.c} implies $\nu^* = \omega^*$ due to $D \succ 0$. Therefore, $\omega_j^* = \zeta$ holds for $\forall j \in\mathcal{N}$. 
By summing \eqref{eq:UEDP_primalfeas.h} up over all the buses in $\mathcal{N}$, we attain 
\beq\label{eq:omegaeq}
\begin{aligned}
& \mathbf{1}^T(q^* -d - D \omega^* - CB \tilde \theta^*)  \\
= \ & - \mathbf{1}^T D \omega^*   = - \zeta \cdot \mathbf{1}^T D \mathbf{1} =0
\end{aligned}
\eeq
where the first equality follows from \eqref{eq:UEDP_primalfeas.f} and $\mathbf{1}^T C = 0$. \eqref{eq:omegaeq} enforces $\zeta=0$ and $\omega^*=\zeta \cdot \mathbf{1} = 0$.

We are now ready to prove Theorem~\ref{teo:planner} by contradiction.

\subsubsection*{Necessary condition}
Given an optimal solution $(p^*,q^*,\omega^*, \tilde \theta^*)$ to \eqref{eq:uedp} with $\omega^* = 0 $, we assume $q^*$ is not optimal with respect to \eqref{eq:edp2}. 
It means that there exists an optimal solution $q^{\circ} \neq q^*$ to \eqref{eq:edp2} that satisfies \eqref{eq:edp2.b}-\eqref{eq:edp2.c} and meanwhile achieves a strictly smaller objective value, i.e., $ \mathbf{1}^T J(p^{\circ}) <  \mathbf{1}^T J(p^{*})$. Define $p^{\circ} := q^{\circ}$, $\omega^{\circ}: = 0$ and $ \tilde \theta^{\circ} := C^TL^{\dagger}(q^{\circ}- d)$. We can readily observe that $(p^{\circ}, q^{\circ},\omega^{\circ}, \tilde \theta^{\circ})$ satisfies \eqref{eq:uedp.b}-\eqref{eq:uedp.e} and is thus feasible for \eqref{eq:uedp} with a strictly smaller objective value than $(p^*,q^*,\omega^*, \tilde \theta^*)$. 
However, this contradicts the fact that $(p^*,q^*,\omega^*, \tilde \theta^*)$ is optimal with respect to \eqref{eq:uedp}. As a result, $q^*$ is an optimal solution to \eqref{eq:edp2}.

\subsubsection*{Sufficient condition}
Given an optimal solution $q^*$ to \eqref{eq:edp2}, we assume $(p^*,q^*,\omega^*, \tilde \theta^*)$ with $p^*=q^*$, $\omega^*=0$ and $\tilde \theta^*=C^TL^{\dagger}(q^{*}- d)$ is not optimal with respect to \eqref{eq:uedp}. It means that there exists an optimal solution $(p^{\circ},q^{\circ},\omega^{\circ},\tilde \theta^{\circ}) \neq (p^*,q^*,\omega^*, \tilde \theta^*)$ to \eqref{eq:uedp} that satisfies \eqref{eq:uedp.b}-\eqref{eq:uedp.e} and meanwhile achieves a strictly smaller objective value, i.e., $\mathbf{1}^T J(p^{\circ}) + \frac{1}{2}{\omega^{\circ}}^T d \omega^{\circ} = \mathbf{1}^T J(p^{\circ}) <  \mathbf{1}^T J(p^{*}) =  \mathbf{1}^T J(p^{*}) + \frac{1}{2}{\omega^{*}}^T d \omega^{*}$, where $\omega^{\circ} = 0$ holds from its optimality.
Since $q^{\circ}$ satisfies \eqref{eq:uedp.c}-\eqref{eq:uedp.d}, it automatically satisfies \eqref{eq:edp2.b}-\eqref{eq:edp2.c}, and is therefore feasible for \eqref{eq:edp2} with a strictly smaller objective value than $q^*$. 
However, this contradicts the fact that $q^*$ is optimal with respect to \eqref{eq:edp2}. As a result, $(p^*,q^*,\omega^*, \tilde \theta^*)$ with $p^*=q^*$, $\omega^*=0$ and $\tilde \theta^*=C^TL^{\dagger}(q^{*}- d)$ is an optimal solution to \eqref{eq:uedp}.

\subsection{Proof of Theorem \ref{teo:eqm}}   \label{apx:proof_of_teo_eqm}


Define $\phi:= (p,q,\omega,\tilde \theta, \alpha, \lambda,\eta,\nu )\backslash(z,\sigma)$ to be the variables that are reduced and updated by taking their optimizers based on
\eqref{eq:gen_gradient_general.b}, \eqref{eq:swing_dynamics_Lagrangian_interpretation.a} and \eqref{eq:market_gradient_general.b}.

\subsubsection*{Necessary condition}
We first ignore the particular projection to guarantee $\eta\ge 0$.
An equilibrium $(z^*,\sigma^*)$ of the grid-market-participant loop \eqref{eq:grid-market-participant_loop} means
\bseq\label{eq:eq_condition}
\begin{align}
\label{eq:eq_condition.a}
& \nabla_z \hat L(z^*,\sigma^*) = 0 \ , \\
\label{eq:eq_condition.b}
& \nabla_{\sigma} \hat L(z^*,\sigma^*) = 0 \ ,
\end{align}
with 
\beq\nonumber
\hat L (z^*,\sigma^*) : = L(z^*,\sigma^*,\phi(z^*,\sigma^*)) \ ,
\eeq
where $\phi(z,\sigma)$ contains the corresponding maximizers or minimizers of $L(z,\sigma,\phi)$, always satisfying
\beq\label{eq:eq_condition.c}
\nabla_{\phi} L(z^*,\sigma^*,\phi(z^*,\sigma^*)) = 0 \ ,
\eeq
\eseq
due to the finiteness of $\hat L$ by Assumption~\ref{ass:finiteness}.

Let $\phi^*: = \phi(z^*,\sigma^*)$ be the short hand.
Note that \eqref{eq:eq_condition.a}, \eqref{eq:eq_condition.b} imply 
\begin{small}
\bseq\label{eq:eq_condition_transform}
\begin{align}
\label{eq:eq_condition_transform.a}
&    \nabla_z \hat L(z^*,\sigma^*) = \left(\nabla_z  L + {\frac{\partial \phi}{\partial z} }^T \nabla_{\phi} L \right) \Bigg\vert_{(z^*,\sigma^*,\phi^*)}= 0 \ , \\
\label{eq:eq_condition_transform.b}
&    \nabla_{\sigma} \hat L(z^*,\sigma^*) =  \left( \nabla_{\sigma}  L + {\frac{\partial \phi}{\partial \sigma} }^T \nabla_{\phi} L \right) \Bigg \vert_{(z^*,\sigma^*,\phi^*)} = 0 \ ,
\end{align}
\eseq
\end{small}%
where $\frac{\partial \phi}{\partial z}$ and $\frac{\partial \phi}{\partial \sigma}$ contain the (sub)gradients of $\phi$ with respect to $z$ and $\sigma$, respectively.
\eqref{eq:eq_condition.c} and \eqref{eq:eq_condition_transform} jointly lead to
\beq\label{eq:eq_stationarity}
   \nabla_z  L(z^*,\sigma^*,\phi^*) = 0 ~\textrm{and}~
   \nabla_{\sigma}  L(z^*,\sigma^*,\phi^*) =0 \ .
\eeq
Indeed, \eqref{eq:eq_condition.c}, \eqref{eq:eq_stationarity} equate with all the stationarity conditions and equality feasibility conditions in the KKT conditions \eqref{eq:UEDP_KKT}. 

We next consider the effect of the projection for $\eta \ge 0$. If $\eta$ is contained in $\sigma$, we have
\beq
\left[ \nabla_{\eta} \hat L(z^*,\sigma^*) \right]^+_{\eta} =\left[ H^T(q^*-d) -F\right]^+_{\eta} =0 \ ,
\eeq
which suggests for any $e^{\rm th}$ elements of these vectors  either $\left(H^T(q^* -d)\right)_e =F_e$ with $\eta_e^* \ge 0$ or $\left(H^T(q^* -d)\right)_e <F_e$ with $\eta_e^* =0$. In either case, the inequality feasibility conditions \eqref{eq:UEDP_primalfeas.g}, \eqref{eq:UEDP_dualfeas.i} and the complementary slackness condition \eqref{eq:UEDP_complementary.j} are simultaneously guaranteed.

If $\eta$ is contained in $\phi$, it follows from \eqref{eq:market_gradient_general.b} and Assumption~\ref{ass:finiteness} that
\beq
 \nabla_{\eta}  L(z^*,\sigma^*, \phi^*)  \le 0 ~ \textrm{and}~ \diag(\eta^*) \nabla_{\eta}  L(z^*,\sigma^*, \phi^*) =0
\eeq
should hold to guarantee the finiteness of $\hat L$ in the presence of $\eta  \ge 0$. Therefore, the inequality feasibility conditions \eqref{eq:UEDP_primalfeas.g}, \eqref{eq:UEDP_dualfeas.i} and the complementary slackness condition \eqref{eq:UEDP_complementary.j} are also met.

All the KKT conditions \eqref{eq:UEDP_KKT} are satisfied in any case and thus $(z^*,\sigma^*,\phi^*)$ is an optimal solution to the planner's problem \eqref{eq:uedp}.

\subsubsection*{Sufficient condition}

Given an optimal solution $(z^*,\sigma^*,\phi^*)$ to the planner's problem \eqref{eq:uedp} that satisfies the KKT conditions \eqref{eq:UEDP_KKT}, if $\eta$ is contained in $\sigma$, all the stationarity conditions and equality feasibility conditions basically equate with \eqref{eq:eq_condition.c}, \eqref{eq:eq_stationarity}, and thus \eqref{eq:eq_condition.a}, \eqref{eq:eq_condition.b}, for all the variables except $\eta$.
Further, the inequality feasibility condition \eqref{eq:UEDP_primalfeas.g} suggests $\nabla_{\eta}\hat L(z^*,\sigma^*) = 0$ immediately if it attains equality; otherwise, the complementary slackness condition \eqref{eq:UEDP_complementary.j} mandates $\eta^* = 0$, which projects the negative $\nabla_{\eta}\hat L(z^*,\sigma^*)$ to zero through the projection $[\cdot]^+_{\eta}$. Therefore, $(z^*,\sigma^*)$ is an equilibrium point of \eqref{eq:grid-market-participant_loop}.

If $\eta$ is contained in $\phi$, indeed the optimality condition of \eqref{eq:market_gradient_general.b} suggests that
\beq
\diag(\eta) \nabla_{\eta}  L(z,\sigma, \phi)= \diag(\eta)\left(H^T(q-d)-F\right) =0
\eeq
should always hold, by Assumption~\ref{ass:finiteness}. In this case, $\hat L$ is independent of $\eta$. Then all the stationarity conditions and equality feasibility conditions in the KKT conditions \eqref{eq:UEDP_KKT} suffice to show \eqref{eq:eq_condition.a}, \eqref{eq:eq_condition.b} and validate that $(z^*,\sigma^*)$ is still an equilibrium point of \eqref{eq:grid-market-participant_loop}.

\subsection{Proof of Theorem \ref{teo:stability}}    \label{apx:proof_of_teo_stability}

We explicitly demonstrate the asymptotic convergence of the grid-market-participant loop \eqref{eq:grid-market-participant_loop}. 
Denote the largest invariant set between the on-off switches of the projection $[\cdot]^+_{\eta}$ as
\beq\label{eq:largest_invariant_set}
\mathbb{S}:=\bigg\{(z,\sigma) \ | \  \dot V(z(t),\sigma(t))  \equiv 0, \ t\in\mathbb{R}_{\ge0}\backslash \mathbb{T}  \bigg\} \ ,
\eeq
where $V(z,\sigma)$ is a real-valued Lyapunov function and $\mathbb{T}$ consists of all the time epochs when the projection switches between on and off.
The whole proof boils down to three steps:
\begin{itemize}
	\item \textbf{Step 1}: Each trajectory $(z(t),\sigma(t))$ converges to the largest invariant set $\mathbb{S}$.
		\item \textbf{Step 2}: Any trajectory $(z(t),\sigma(t))$ in the largest invariant set $\mathbb{S}$ is an equilibrium of the closed-loop system \eqref{eq:grid-market-participant_loop}, i.e., $\mathbb{S}\subset \mathbb{E}$.
			\item \textbf{Step 3}: Each trajectory $(z(t),\sigma(t))$ literally converges to a single equilibrium point of the closed-loop system \eqref{eq:grid-market-participant_loop}.
\end{itemize}

We next prove each individual step.
 
\noindent
 \textbf{Step 1}:
 
The grid-market-participant loop \eqref{eq:grid-market-participant_loop} is essentially implementing a (projected) saddle flow on $\hat L(z,\sigma)$, and 
each equilibrium point $(z^*,\sigma^*)$ is therefore a min-max saddle point of $\hat L(z,\sigma)$ (within a specified domain).
Consider the following quadratic Lyapunov function for $V(\cdot)$: 
\beq\label{eq:quadratic_Lyap} 
 V(z,\sigma)
 := \  \frac{1}{2}
\begin{bmatrix}
z-z^* \\ \sigma-\sigma^* 
\end{bmatrix}^T
\begin{bmatrix}
	T^z & \\
	& T^\sigma   
\end{bmatrix}
\begin{bmatrix}
z-z^* \\ \sigma-\sigma^* 
\end{bmatrix} \ ,
\eeq
where $(z^*,\sigma^*)$ is one arbitrary equilibrium of the closed-loop system \eqref{eq:grid-market-participant_loop}.
The time derivative of $V(z,\sigma)$ along the trajectory of $(z(t),\sigma(t))$ is given by
\begin{small}
\bseq
\begin{align}\nonumber
	&\dot V(z,\sigma) \\
	\label{eq:dotVQB.a}
	=  \ & 
(z-z^*)^T T^z \dot z + (\sigma-\sigma^*)^T T^{\sigma} \dot \sigma \\
	\label{eq:dotVQB.b}
 = \ & 	- (z-z^*)^T \nabla_z \hat L(z,\sigma)
+
(\sigma-\sigma^*)^T  \left[\nabla_{\sigma} \hat L(z,\sigma) \right]_{\eta}^+  \\
	\label{eq:dotVQB.c}
\le  \ & - (z-z^*)^T \nabla_z \hat L(z,\sigma)
+
(\sigma-\sigma^*)^T  \nabla_{\sigma} \hat L(z,\sigma)   \\
	\label{eq:dotVQB.d}
\le  \ & \hat L(z^*,\sigma) -  \hat L(z,\sigma)    +  \hat L(z,\sigma) - \hat L(z,\sigma^*)   \\
	\label{eq:dotVQB.e}
	=  \ & \underbrace{ \hat L(z^*,\sigma) -  \hat L(z^*,\sigma^*)    }_{\le 0}   + \underbrace{ \hat L(z^*,\sigma^*) -  \hat L(z,\sigma^*)     }_{\le 0} \ .
\end{align}
\eseq\end{small}%
\eqref{eq:dotVQB.b} applies \eqref{eq:grid-market-participant_loop}. \eqref{eq:dotVQB.c} uses the non-expansive property of the projection $[\cdot]^+_{\eta}$ demonstrated in \eqref{eq:nonexpansive}.
\eqref{eq:dotVQB.d} results from the convexity (resp. concavity) of $ \hat L(z,\sigma)$ in $z$ (resp. $\sigma$). \eqref{eq:dotVQB.e} finally follows from the saddle property of the equilibrium point $(z^*,\sigma^*)$.

Since $V(z,\sigma)$ is radially unbounded, $\dot V(z,\sigma)\le 0$ indicates that all the trajectories $(z(t),\sigma(t))$ remain bounded. 
It then follows from the invariance principle for Caratheodory systems \cite{bacciotti2006nonpathological} that $(z(t),\sigma(t))$ converges to the largest invariant set $\mathbb{S}$.

\noindent
\textbf{Step 2}:

For any trajectory $(z(t),\sigma(t)) \in \mathbb{S}$, $\dot V (z,\sigma) \equiv 0$ enforces \eqref{eq:dotVQB.e} to be zero with
\beq
 \hat L(z(t),\sigma^*) \equiv  \hat L(z^*,\sigma^*)~\textrm{and}~  \hat L(z^*,\sigma(t))  \equiv  \hat L(z^*,\sigma^*)  \ .
\eeq
By Assumption~\ref{ass:observability}, we have $\dot z, \dot \sigma \equiv 0$ for any trajectory $(z(t),\sigma(t))$ in the largest invariant set $\mathbb{S}$, which is therefore an equilibrium of the closed-loop system \eqref{eq:grid-market-participant_loop}, i.e., $\mathbb{S}\subset \mathbb{E}$.

\noindent
\textbf{Step 3}:

We now show any trajectory $(z(t),\sigma(t))$ indeed converges to one single equilibrium.
First of all, along any trajectory $(z(t),\sigma(t))$, $V(z,\sigma)$ is non-increasing in $t$. Since $V(z,\sigma)$ is lower bounded given its quadratic form, there exists an infinite sequence of time epochs $\{t_k,k=1,2,\dots\}$ such that with $k \rightarrow \infty$, we have $\dot V(z(t_k),\sigma(t_k)) \rightarrow 0$, i.e., $(z(t_k),\sigma(t_k)) \rightarrow (\hat z^*,\hat \sigma^*)  \in \mathbb{S} \subset \mathbb{E}$. We use this specific $(\hat z^*,\hat \sigma^*) $ as the equilibrium point in the definition of $V(z,\sigma)$, which implies $V(z(t),\sigma(t)) \rightarrow V(\hat z^*,\hat \sigma^*) =0$. Due to the continuity of $V(z,\sigma)$, $(z(t),\sigma(t)) \rightarrow (\hat z^*, \hat \sigma^*)$ is enforced, which therefore suggests that $(z(t),\sigma(t)) $ indeed converges to one single equilibrium point in $\mathbb{S}\subset \mathbb{E}$. \qed


\subsection{Proof of Proposition~\ref{prop:observability_quantity}}\label{apx:proof-of-prop_observability}

Given $\hat L$ in \eqref{eq:reduced_Lagrangian_quantity} and an arbitrary trajectory $(z(t),\sigma(t))$ of the grid-market-participant loop \eqref{eq:grid-market-participant_loop} that satisfies
\beq\label{eq:QBinvariantcond1}
 \hat L(z(t),\sigma^*) \equiv  \hat L(z^*,\sigma^*)
\eeq
and 
\beq\label{eq:QBinvariantcond2}
 \hat L(z^*,\sigma(t))  \equiv  \hat L(z^*,\sigma^*)  \ ,
\eeq
differentiating \eqref{eq:QBinvariantcond1} with respect to time yields
\beq
\begin{aligned}
0 \equiv & \left(\nabla_z \hat L(z(t),\sigma^*) \right)^T\dot z \\
\equiv & - \left(\nabla_z \hat L(z(t),\sigma^*) \right)^T   {T^{z}}^{-1}\nabla_z \hat L(z(t),\sigma^*)  \ ,
\end{aligned}
\eeq
which indicates $\nabla_z \hat L(z(t),\sigma^*) \equiv 0 $ due to ${T^{z}}^{-1} \succ 0$. 
Given $\sigma^*$, the fact of $\nabla_p \hat L(z(t),\sigma^*) \equiv 0 $ enforces $p(t) \equiv  p^*$, due to the monotonicity of $\nabla J(\cdot)$.


Similarly, differentiating \eqref{eq:QBinvariantcond2} with respect to time gives
\beq\label{eq:QBinvariantcond2_detail}
\begin{aligned}
	0 \equiv & \left(\nabla_{\sigma} \hat L(z^*,\sigma(t)) \right)^T\dot \sigma \\
	\equiv & \left(\nabla_{\sigma} \hat L(z^*,\sigma(t)) \right)^T  {T^{\sigma}}^{-1}\left[ \nabla_{\sigma} \hat L(z^*,\sigma(t)) \right]^+_{\eta} \ ,
\end{aligned}
\eeq
which still enforces $\nabla_{\nu} \hat L(z^*,\sigma(t))  \equiv 0$, $\nabla_{\lambda} \hat L(z^*,\sigma(t))  \equiv 0$ and meanwhile
\beq\label{eq:QBinvariantcond2_projection}
\left(\nabla_{\eta} \hat L(z^*,\sigma(t)) \right)^T  {T^{\eta}}^{-1}\left[ \nabla_{\eta} \hat L(z^*,\sigma(t)) \right]^+_{\eta} \equiv 0 
\eeq
due to ${T^{\sigma}}^{-1} \succ 0$.


Given $p(t)\equiv p^*$, the fact of $\nabla_{\lambda} \hat L(z^*,\sigma(t))  \equiv 0$ implies $\dot \lambda \equiv 0$.
Meanwhile, in terms of \eqref{eq:QBinvariantcond2_projection}, the inner term $H^T(p^*-d)-F$ of the projection $[\cdot]^+_{\eta}$ is a constant vector. Consider an arbitrary $e^{\rm th}$ element of the vector which falls into three cases: (a) $(H^T(p^*-d)-F)_e > 0$; (b) $(H^T(p^*-d)-F)_e = 0$; (c) $(H^T(p^*-d)-F)_e < 0$. In case (a), $\dot \eta_e > 0$ drives $\eta_e(t)$ to infinity, which cannot happen since all the trajectories in the largest invariant set $\mathbb{S}$ are bounded. 
Case (b) immediately implies $\dot \eta_e \equiv 0$. In case (c), since ${T^{\eta}}^{-1} \succ 0$ is diagonal, $\left(\nabla_{\eta} \hat L(z^*,\sigma(t))\right)_e < 0$ enforces $\left[\left( \nabla_{\eta} \hat L(z^*,\sigma(t)) \right)_e\right]^+_{\eta_e} \equiv 0$, i.e., $\dot \eta_e \equiv 0$, in order for \eqref{eq:QBinvariantcond2_projection} to hold.
As a result, $\dot \eta \equiv 0$ is guaranteed.
On top of $\dot \lambda, \dot \eta \equiv 0$, $T^p \dot p = \lambda(t)\cdot \mathbf{1} -H \eta(t) - \nu(t) -\nabla_p J(p^*) \equiv 0 $ enforces $\dot \nu \equiv 0$, or $\nu(t) \equiv \bar \nu$, where $\bar \nu$ is constant.
The fact of $\nabla_{\nu} \hat L(z^*,\sigma(t))  \equiv 0$ suggests
\beq
 p^* - d -D \bar \nu - CB\tilde \theta^* \equiv 0 \ .
\eeq
Due to $p^* - d -D \nu^* - CB\tilde \theta^* = 0$ from \eqref{eq:UEDP_primalfeas.e}, \eqref{eq:UEDP_primalfeas.h}, $\nu(t)\equiv \bar \nu = \nu^* = 0$ follows immediately due to $D \succ 0$, which further implies $\dot {\tilde \theta} \equiv 0$.
This completes the proof of $\dot z, \dot \sigma \equiv 0$.

\subsection{Proof of Theorem \ref{teo:eqm_reg}}   \label{apx:proof_of_teo_eqm_reg}

Note that the closed-loop system \eqref{eq:swingdynamics}, \eqref{eq:gradient_play_gen_price_bidding_alternative}, \eqref{eq:regularized_market_graident_play} is defined on $z=(\hat q, \tilde \theta)$ and $\sigma=(\lambda,\omega,\alpha,\eta)$.

\subsubsection*{Necessary condition}
Note that the reduced variables $p^*, q^*,\nu^*$ are recovered to satisfy $p^*=\left( \nabla J\right)^{-1}(\lambda^*\cdot \mathbf{1} -H \eta^* -\omega^*)$, $ q^* = \frac{1}{\rho} (\lambda^*\cdot \mathbf{1} -H \eta^* -\omega^* - \alpha^*) + \hat q^*$ and $\nu^* = \omega^*$. Therefore, \eqref{eq:UEDP_stationarity.c} is self-evident.
Given an equilibrium $(z^*,\sigma^*)$ of the closed-loop system \eqref{eq:swingdynamics}, \eqref{eq:gradient_play_gen_price_bidding_alternative}, \eqref{eq:regularized_market_graident_play}, $\dot{\hat q} = 0$ enforces \eqref{eq:UEDP_stationarity.b} and further indicates $q^* = \hat q^*$ and \eqref{eq:UEDP_stationarity.a}.
$\dot{\tilde \theta},\dot\omega,\dot \lambda =0$ similarly suffice to show $\omega^*=\nu^* = 0$, suggesting \eqref{eq:UEDP_stationarity.d}, from the previous analysis \eqref{eq:omegaeq}.
Besides, $\dot \alpha =0$ implies $q^* = p^*$, i.e., \eqref{eq:UEDP_primalfeas.e}. \eqref{eq:UEDP_primalfeas.f}, \eqref{eq:UEDP_primalfeas.h} then immediately follow from $\dot \lambda,\dot \omega =0$, respectively.
The definition of the projection $[\cdot]^+_{\eta}$ and $\dot \eta =0$ can guarantee \eqref{eq:UEDP_primalfeas.g}, \eqref{eq:UEDP_dualfeas.i} and \eqref{eq:UEDP_complementary.j}, based on the previous discussion in Appendix~\ref{apx:proof_of_teo_eqm}.
So far the KKT conditions \eqref{eq:UEDP_KKT} of the planner's problem \eqref{eq:uedp} have all been met from $\dot z,\dot \sigma=0$, indicating its optimality.

\subsubsection*{Sufficient condition}
Given the KKT conditions \eqref{eq:UEDP_KKT} of the planner's problem \eqref{eq:uedp}, we now aim to attain $\dot z, \dot \sigma=0$.
First of all, due to \eqref{eq:UEDP_stationarity.b}, real dispatch $q^*$ defined in \eqref{eq:regularized_market_graident_play.a} reduces to $q^* = \hat q^*$, immediately suggesting $\dot{\hat q} =0$.
\eqref{eq:UEDP_stationarity.a}, \eqref{eq:UEDP_stationarity.b} and \eqref{eq:UEDP_primalfeas.e} jointly enforce $\dot \alpha =0$.
$\omega^*=0$, \eqref{eq:UEDP_primalfeas.f} and \eqref{eq:UEDP_primalfeas.h} imply $\dot{\tilde \theta}, \dot \lambda, \dot \omega= 0$, respectively.
Given \eqref{eq:UEDP_primalfeas.g}, $\dot \eta = 0$ holds immediately if it attains equality; otherwise, \eqref{eq:UEDP_complementary.j} enforces $\eta^* = 0$, which still indicates $\dot \eta = 0$ due to the projection $[\cdot]^+_{\eta}$.
So far we have shown $\dot z,\dot \sigma =0$, indicating that $(z^*,\sigma^*)$ is an equilibrium of the closed-loop system \eqref{eq:swingdynamics}, \eqref{eq:gradient_play_gen_price_bidding_alternative}, \eqref{eq:regularized_market_graident_play}, starting from one initial point in $\mathbb{I}$.

\subsection{Proof of Theorem \ref{teo:stability_reg}}    \label{apx:proof_of_teo_stability_reg}

Recall the quadratic form of generation cost function \eqref{eq:quadratic_gen_function}, parameterized by $c:=(c_j,j\in\mathcal{N})$ and $\bar c:=(\bar c_j, j\in\mathcal{N})$. We further define $c^{-1}: =(c_j^{-1},j\in\mathcal{N})$. Therefore, \eqref{eq:gradient_play_gen_price_bidding_alternative} can be more explicitly expressed as 
\beq\nonumber
T^{\alpha} \dot \alpha = \frac{1}{\rho}(\lambda\cdot \mathbf{1} - H \eta -\omega -\alpha) + \hat q - \diag(c^{-1})(\lambda\cdot \mathbf{1} - H \eta -\omega - \bar c) .
\eeq
With the closed-loop system \eqref{eq:swingdynamics}, \eqref{eq:gradient_play_gen_price_bidding_alternative}, \eqref{eq:regularized_market_graident_play} defined on $z = (\hat q, \tilde \theta)$ and $\sigma= (\lambda,\omega,\alpha,\eta)$, we can define a square matrix $W \in \mathbb{R}^{(3|\mathcal{N}|+3|\mathcal{E}|+1)\times(3|\mathcal{N}|+3|\mathcal{E}|+1)}$ to summarize the right-hand-side linear dependence of the differential equations \eqref{eq:swingdynamics}, \eqref{eq:gradient_play_gen_price_bidding_alternative}, \eqref{eq:regularized_market_graident_play} on $(z,\sigma)$ such that we attain a more compact form:
\beq
\begin{bmatrix}
	T^z & \\
	& T^{\sigma}
\end{bmatrix}
\begin{bmatrix}
\dot z \\ \dot \sigma 
\end{bmatrix} 
= \left[W 
\begin{bmatrix}
	z \\ \sigma
\end{bmatrix}  + \beta \right]^{+}_{\eta}   \ ,
\eeq 
with a constant input $\beta \in \mathbb{R}^{3\vert \mathcal{N}\vert + 3\vert \mathcal{E}\vert +1}$ given by 
\beq
\beta:=
\begin{bmatrix}
    0 & 0 & \mathbf 1^T d  &  -d^T &  \bar c^T \diag (c^{-1}) & -d^T H - F^T 
\end{bmatrix}^T  .
\eeq

We focus on the trajectories $(z(t),\sigma(t))$ that start with initial points in $\mathbb{I}$. Note that we still have the equilibrium set and the largest invariant set denoted as $\mathbb{E}$ and $\mathbb{S}$, respectively, from \eqref{eq:equilibrium_set}, \eqref{eq:largest_invariant_set}.
Given the condition $\rho\in \left(0  ,  \inf_{j\in\mathcal{N}} 4 c_j\right) $, the whole proof still follows the same three steps in the proof of Theorem~\ref{teo:stability} in Appendix~\ref{apx:proof_of_teo_stability}.
We now show each individual step.


\noindent
\textbf{Step 1}: 

We still adopt the standard quadratic Lyapunov function \eqref{eq:quadratic_Lyap}, but now
$(z^*,\sigma^*)$ is one arbitrary equilibrium of the closed-loop system \eqref{eq:swingdynamics}, \eqref{eq:gradient_play_gen_price_bidding_alternative}, \eqref{eq:regularized_market_graident_play}.
We can acquire the time derivative of $V(z,\sigma)$ along the trajectory of $(z(t),\sigma(t))$ as
\bseq
\begin{small}
\begin{align}\nonumber
&\dot V(z,\sigma) \\
\label{eq:dotVPB.a}
=  \ & 
\begin{bmatrix}
z-z^* \\ \sigma-\sigma^* 
\end{bmatrix}^T
\begin{bmatrix}
T^z \dot z \\ T^{\sigma}\dot \sigma
\end{bmatrix} \\
\label{eq:dotVPB.b}
= \ & \begin{bmatrix}
z-z^* \\ \sigma-\sigma^* 
\end{bmatrix}^T
\left[W 
\begin{bmatrix}
z \\ \sigma
\end{bmatrix} + \beta \right]^{+}_{\eta} \\
\label{eq:dotVPB.c}
\le \ & \begin{bmatrix}
z-z^* \\ \sigma-\sigma^* 
\end{bmatrix}^T
\left( W 
\begin{bmatrix}
z \\ \sigma
\end{bmatrix}+ \beta \right)	\\
\label{eq:dotVPB.d}
= \ & \begin{bmatrix}
z-z^* \\ \sigma-\sigma^* 
\end{bmatrix}^T
W 
\begin{bmatrix}
z-z^* \\ \sigma-\sigma^*
\end{bmatrix}  - (\eta- \eta^*)^T \psi^* \\
\label{eq:dotVPB.e}
\le \ & \begin{bmatrix}
z-z^* \\ \sigma-\sigma^* 
\end{bmatrix}^T
W 
\begin{bmatrix}
z-z^* \\ \sigma-\sigma^*
\end{bmatrix}  \\
\label{eq:dotVPB.f}
= \ & - (\sigma-\sigma^*)^T W_{\sigma} (\sigma-\sigma^*) \ ,
\end{align}%
\end{small}%
\eseq%
where $\psi^* \in \mathbb{R}_{\ge0}^{2|\mathcal{E}|}$ denotes a complementary vector variable and $W_\sigma \in\mathbb{R}^{(2|\mathcal{N}|+ 2|\mathcal{E}|+1)\times(2|\mathcal{N}|+ 2|\mathcal{E}|+1)}$, with its explicit expression in \eqref{eq:Wz}, is a symmetric matrix re-arranged from the submatrix of $-W$ that contains all the rows and columns with regard to $\sigma$. Note that \eqref{eq:dotVPB.d} follows from the equilibrium condition
\beq
W
\begin{bmatrix}
z^* \\ \sigma^*
\end{bmatrix}
  + \beta +
\begin{bmatrix}
 0^{3|\mathcal{N}|+|\mathcal{E}|+ 1}  \\ \psi^* 
\end{bmatrix} 
=0
\eeq
along with the complementarity condition
\beq
0\le \eta^* \perp \psi^* \ge 0 \ .
\eeq
\eqref{eq:dotVPB.e} results from the fact $(\eta-\eta^*)^T\psi^* = \eta^T \psi^* - {\eta^*}^T \psi^* = \eta^T \psi^* \ge 0$. \eqref{eq:dotVPB.f} is attained with all the terms regarding $z-z^*$ canceled out due to the specific structure of $W$.

\begin{table*}
\beq\label{eq:Wz}
W_\sigma :=  \begin{bmatrix}
	|\mathcal{N}|\rho^{-1} & - \rho^{-1}  \cdot \mathbf{1}^T &   \frac{1}{2}\cdot{c^{-1}}^T  - \rho^{-1}\cdot  \mathbf{1}^T  & - \rho^{-1} \cdot \mathbf 1^T H  \\
	- \rho^{-1}\cdot   \mathbf{1} & \rho^{-1}\cdot I + D   & \rho^{-1}\cdot  I - \frac{1}{2}\cdot \diag(c^{-1}) &    \rho^{-1}\cdot  H  \\
	\frac{1}{2}\cdot {c^{-1}}  - \rho^{-1} \cdot \mathbf{1}  &  \rho^{-1}\cdot  I - \frac{1}{2}\cdot \diag(c^{-1})     &   \rho^{-1}\cdot  I &   \left(  \rho^{-1} \cdot  I - \frac{1}{2}\cdot \diag(c^{-1}) \right) H  \\
	-\rho^{-1} \cdot  H^T  \mathbf 1  & \rho^{-1} \cdot  H^T  & H^T \left(  \rho^{-1}\cdot  I - \frac{1}{2}\cdot \diag(c^{-1}) \right)   & \rho^{-1}\cdot  H^T H
\end{bmatrix}
\eeq
\end{table*}

We next claim the positive semidefiniteness of $W_\sigma$. Let
\beq
R:= I- \frac{\rho}{2} \cdot \diag(c^{-1})
\eeq 
be a diagonal matrix and we can re-express $W_\sigma$ as 
\begin{small}
\beq\label{eq:innerWz}
W_\sigma = \rho^{-1} Q^T
\underbrace{
\begin{bmatrix}
	I & - I  & -R& -I \\
- I & I +\rho D & R& I   \\
-R & R  & I & R \\
-I & I & R& I 
\end{bmatrix}}_{=: W_\sigma^{\rm in}}
Q
\eeq
\end{small}%
with
\beq
Q:=\blockdiag(\mathbf{1},I,I,H) \in \mathbb{R}^{4|\mathcal{N}|\times(2|\mathcal{N}|+ 2|\mathcal{E}|+1)} \ .
\eeq
The inner matrix $W_\sigma^{\rm in}\in \mathbb{R}^{4|\mathcal{N}|\times 4|\mathcal{N}|}$ of \eqref{eq:innerWz} can be further block diagonalized by the following bijective linear transformation $P\in \mathbb{R}^{4|\mathcal{N}|\times 4|\mathcal{N}|}$:
\begin{small}
\beq\label{eq:diagWz}
W_\sigma \\
= \rho^{-1} Q^T P^T
\begin{bmatrix}
	I & 0 & 0 &  0 \\
	0 & \rho D& 0 &  0\\
	0 & 0 &  I-R^2 &  0 \\
	0 & 0 & 0 &  0
\end{bmatrix}
PQ
\eeq
\end{small}%
with
\begin{small}
\beq
P : =
\begin{bmatrix}
	I&I& R &I \\
	0& I & 0 &0\\
    0& 0 & I &0\\
	0 & 0 & 0 & I
\end{bmatrix}^{-1} \ .
\eeq
\end{small}%
Note that $I-R^2$ is still diagonal and it is easy to verify that any arbitrary $\rho\in \left(0  ,  \inf_{j\in\mathcal{N}}   4 c_j \right) $ guarantees its positive definiteness.
$W_\sigma \succeq 0 $ then follows from $I, \rho D ,I- R^2\succ 0$.

Therefore, following \eqref{eq:dotVPB.f}, we arrive at 
\beq
\dot V(z,\sigma) \le - (\sigma-\sigma^*)^T W_\sigma (\sigma-\sigma^*) \le 0 \ .
\eeq
Since $V(z,\sigma)$ is radially unbounded, $\dot V(z,\sigma) \le 0$ indicates that all the trajectories $(z(t),\sigma(t))$ remain bounded. 
It then follows from the invariance principle for Caratheodory systems that $(z(t),\sigma(t))$ converges to the largest invariant set $\mathbb{S}$.

\noindent
\textbf{Step 2}:

For an arbitrary trajectory $(z(t),\sigma(t)) \in \mathbb{S}$, $\dot V(z,\sigma) \equiv 0$ basically enforces $(\sigma-\sigma^*)^TW_\sigma(\sigma-\sigma^*) \equiv 0$. In light of the structure of $W_\sigma$ in \eqref{eq:innerWz}, its semidefiniteness results from the nontrivial null space of $W_\sigma^{\rm in}$ and $Q$.
As a result, we can characterize the largest invariant set $\mathbb{S}$ as the union of the following three sets:
\begin{small}
\beq\nonumber
\mathbb{S}_1:=\Big\{(z,\sigma) \  | \   \sigma(t) \equiv \sigma^* \Big\} 
\eeq
\beq\nonumber
\begin{aligned}
 \mathbb{S}_2:=\Big\{(z,\sigma) \  | \  &  H (\eta(t)-\eta^*)\equiv0 , \  \lambda(t) \equiv \lambda^*,  \\   & \omega(t)\equiv \omega^*, \ \alpha(t)\equiv \alpha^*   \Big\} 
\end{aligned}
\eeq
\beq\nonumber
\begin{aligned}
\mathbb{S}_3:=\Big\{(z,\sigma) \  | \ &  (\lambda(t)-\lambda^*)\cdot\mathbf{1} \equiv H (\eta(t)-\eta^*) ,   \\ &  \omega(t)\equiv \omega^*, \ \alpha(t)\equiv \alpha^*   \Big\} 
\end{aligned}
\eeq\end{small}%

We next claim $\mathbb{S}_2 \equiv \mathbb{S}_3$ by first showing $\mathbf{1}^T H =0$. 
Recall $H = [(BC^TL^{\dagger})^T, - (BC^TL^{\dagger})^T]$ and $L= CBC^T$. By definition of the Moore-Penrose inverse, it is straightforward to have $L^{\dagger}\mathbf{1} =0$ due to
\begin{small}
\beq
\begin{aligned}
	L^{\dagger} \mathbf{1} =&\ (CBC^T)^{\dagger}(CBC^T)(CBC^T)^{\dagger} \mathbf{1} \\
=&\ (CBC^T)^{\dagger}((CBC^T)(CBC^T)^{\dagger})^T \mathbf{1}\\
=& \ (CBC^T)^{\dagger}{(CBC^T)^{\dagger}}^T CB \underbrace{C^T\mathbf{1} }_{=0} \ .
\end{aligned}
\eeq\end{small}%
$\mathbf{1}^T H =0$ then follows immediately.

In $\mathbb{S}_3$, $(\lambda(t)-\lambda^*)\cdot\mathbf{1} \equiv H (\eta(t)-\eta^*)$ leads to 
\beq
\underbrace{\mathbf{1}^T \cdot (\lambda(t)-\lambda^*)\cdot\mathbf{1} }_{ =|\mathcal{N}|(\lambda(t)-\lambda^*)}\equiv \underbrace{\mathbf{1}^T H (\eta(t)-\eta^*) }_{ = 0} \ ,
\eeq
which enforces $\lambda(t)\equiv\lambda^*$ and $ H (\eta(t)-\eta^*)\equiv0$, i.e., $\mathbb{S}_3 \subset \mathbb{S}_2$.
Obviously by definition we have $\mathbb{S}_2 \subset \mathbb{S}_3$.
Therefore, $\mathbb{S}_2$ and $\mathbb{S}_3$ are equivalent.

Note that in any of the three sets, we have $\lambda(t) \equiv \lambda^*$,  $\omega(t)\equiv \omega^*$, $\alpha(t)\equiv \alpha^*$ and $H\eta(t)\equiv H\eta^*$, which suffice to guarantee
$\dot{\hat q}, \dot{\tilde{\theta}} \equiv 0 $ immediately, or $\dot z \equiv 0$; recall \eqref{eq:swingdynamics:1}, \eqref{eq:regularized_market_graident_play.b}. 
Then it suggests that the inner term of the projection $[\cdot]^+_{\eta}$ in \eqref{eq:regularized_market_graident_play.d} is a constant vector. For its $e^{\rm th}$ entry, it could be (a) positive; (b) zero; (c) negative. In case (a), $\dot \eta_e > 0$ drives $\eta_e(t)$ to infinity, which contradicts the fact that all the trajectories in the largest invariant set $\mathbb{S}$ are bounded. Case (b) directly implies $\dot \eta_e \equiv 0$. 
In case (c), $\eta_e(t)$ will be driven to stay at 0 that also suggests $\dot \eta_e \equiv 0$.
As a result, $\dot \eta \equiv 0$ holds.\footnote{There might exist a trivial degenerate subspace in case (c) where $H(\eta(t)-\eta^*) \equiv 0$ holds with $\dot \eta \not\equiv 0$. However, in this subspace, if any, the system is still in transient and will eventually converge to $\eta(t)\equiv \eta^*$. Therefore, we exclude it from the characterization of $\mathbb{S}$ for conciseness.  }
So far, $\dot \sigma \equiv 0$ has also been guaranteed. 
Therefore, any trajectory $(z(t),\sigma(t))$ in the largest invariant set $\mathbb{S}$ is an equilibrium of the closed-loop system \eqref{eq:swingdynamics}, \eqref{eq:gradient_play_gen_price_bidding_alternative}, \eqref{eq:regularized_market_graident_play}, i.e., $\mathbb{S}\subset \mathbb{E}$.

\noindent
\textbf{Step 3}:

We now prove any trajectory $(z(t),\sigma(t))$ indeed converges to one single equilibrium for completeness, despite similar procedures to Step 3 in the proof of Theorem \ref{teo:stability}.
We observe that along any trajectory $(z(t),\sigma(t))$, $V(z,\sigma)$ is non-increasing in $t$. Since $V(z,\sigma)$ is lower bounded given its quadratic form, there exists an infinite sequence of time epochs $\{t_k,k=1,2,\dots\}$ such that with $k \rightarrow \infty$, we have $\dot V(z(t_k),\sigma(t_k)) \rightarrow 0$, i.e., $(z(t_k),\sigma(t_k)) \rightarrow (\hat z^*,\hat \sigma^*)  \in \mathbb{S} \subset \mathbb{E}$. We use this specific $(\hat z^*,\hat \sigma^*) $ as the equilibrium point in the definition of $V(z,\sigma)$, which implies $V(z(t),\sigma(t)) \rightarrow V(\hat z^*,\hat \sigma^*) =0$. Due to the continuity of $V(z,\sigma)$, $(z(t),\sigma(t)) \rightarrow (\hat z^*, \hat \sigma^*)$ is enforced, which suggests that $(z(t),\sigma(t)) $ literally converges to one single equilibrium point in $\mathbb{S}\subset \mathbb{E}$.



\bibliographystyle{IEEEtran}
\bibliography{IEEEabrv,mybibfile}

\begin{thebibliography}{10}
\providecommand{\url}[1]{#1}
\csname url@samestyle\endcsname
\providecommand{\newblock}{\relax}
\providecommand{\bibinfo}[2]{#2}
\providecommand{\BIBentrySTDinterwordspacing}{\spaceskip=0pt\relax}
\providecommand{\BIBentryALTinterwordstretchfactor}{4}
\providecommand{\BIBentryALTinterwordspacing}{\spaceskip=\fontdimen2\font plus
\BIBentryALTinterwordstretchfactor\fontdimen3\font minus
  \fontdimen4\font\relax}
\providecommand{\BIBforeignlanguage}[2]{{%
\expandafter\ifx\csname l@#1\endcsname\relax
\typeout{** WARNING: IEEEtran.bst: No hyphenation pattern has been}%
\typeout{** loaded for the language `#1'. Using the pattern for}%
\typeout{** the default language instead.}%
\else
\language=\csname l@#1\endcsname
\fi
#2}}
\providecommand{\BIBdecl}{\relax}
\BIBdecl

\bibitem{you18stabilization}
P.~You, J.~Z.~F. Pang, and E.~Yeung, ``Stabilization of power networks via
  market dynamics,'' in \emph{ACM International Conference on Future Energy
  Systems (e-Energy)}, 2018, pp. 139--145.

\bibitem{kirschen2018fundamentals}
D.~S. Kirschen and G.~Strbac, \emph{Fundamentals of power system
  economics}.\hskip 1em plus 0.5em minus 0.4em\relax John Wiley \& Sons, 2018.

\bibitem{alvarado2001stability}
F.~L. Alvarado, J.~Meng, C.~L. DeMarco, and W.~S. Mota, ``Stability analysis of
  interconnected power systems coupled with market dynamics,'' \emph{IEEE
  Transactions on Power Systems}, vol.~16, no.~4, pp. 695--701, 2001.

\bibitem{machowski1997power}
J.~Machowski, J.~Bialek, J.~R. Bumby, and J.~Bumby, \emph{Power system dynamics
  and stability}.\hskip 1em plus 0.5em minus 0.4em\relax John Wiley \& Sons,
  1997.

\bibitem{wood2013power}
A.~J. Wood, B.~F. Wollenberg, and G.~B. Shebl{\'e}, \emph{Power generation,
  operation, and control}.\hskip 1em plus 0.5em minus 0.4em\relax John Wiley \&
  Sons, 2013.

\bibitem{zhao2016unified}
C.~Zhao, E.~Mallada, S.~Low, and J.~Bialek, ``A unified framework for frequency
  control and congestion management,'' in \emph{IEEE Power Systems Computation
  Conference (PSCC)}, 2016, pp. 1--7.

\bibitem{li2016connecting}
N.~Li, C.~Zhao, and L.~Chen, ``Connecting automatic generation control and
  economic dispatch from an optimization view,'' \emph{IEEE Transactions on
  Control of Network Systems}, vol.~3, no.~3, pp. 254--264, 2016.

\bibitem{mallada2017optimal}
E.~Mallada, C.~Zhao, and S.~Low, ``Optimal load-side control for frequency
  regulation in smart grids,'' \emph{{IEEE Transactions on Automatic Control}},
  vol.~62, no.~12, pp. 6294--6309, 2017.

\bibitem{wang2017distributed1}
Z.~Wang, F.~Liu, S.~H. Low, C.~Zhao, and S.~Mei, ``{Distributed frequency
  control with operational constraints, part I: Per-node power balance},''
  \emph{IEEE Transactions on Smart Grid}, vol.~10, no.~1, pp. 40--52, 2017.

\bibitem{wang2017distributed2}
{Z. Wang, F. Liu, S. H. Low, C. Zhao, and S. Mei}, ``{Distributed frequency
  control with operational constraints, part II: Network power balance},''
  \emph{IEEE Transactions on Smart Grid}, vol.~10, no.~1, pp. 53--64, 2017.

\bibitem{weitenberg2018robust}
E.~Weitenberg, Y.~Jiang, C.~Zhao, E.~Mallada, C.~De~Persis, and F.~D{\"o}rfler,
  ``Robust decentralized secondary frequency control in power systems: Merits
  and tradeoffs,'' \emph{IEEE Transactions on Automatic Control}, vol.~64,
  no.~10, pp. 3967--3982, 2018.

\bibitem{pang2017optimal}
J.~Z. Pang, L.~Guo, and S.~H. Low, ``Optimal load control for frequency
  regulation under limited control coverage,'' in \emph{IREP Symposium}, 2017,
  pp. 1--7.

\bibitem{Dorfler_2016}
F.~Dorfler, J.~W. Simpson-Porco, and F.~Bullo, ``Breaking the hierarchy:
  Distributed control and economic optimality in microgrids,'' \emph{IEEE
  Transactions on Control of Network Systems}, vol.~3, no.~3, p. 241–253, Sep
  2016.

\bibitem{cai2017distributed}
D.~Cai, E.~Mallada, and A.~Wierman, ``Distributed optimization decomposition
  for joint economic dispatch and frequency regulation,'' \emph{IEEE
  Transactions on Power Systems}, vol.~32, no.~6, pp. 4370--4385, 2017.

\bibitem{jin2013designing}
Y.~G. Jin, S.~Y. Lee, S.~W. Kim, and Y.~T. Yoon, ``Designing rule for
  price-based operation with reliability enhancement by reducing the frequency
  deviation,'' \emph{IEEE Transactions on Power Systems}, vol.~28, no.~4, pp.
  4365--4372, 2013.

\bibitem{watts2004influence}
D.~Watts and F.~L. Alvarado, ``The influence of futures markets on real time
  price stabilization in electricity markets,'' in \emph{IEEE Annual Hawaii
  International Conference on System Sciences (HICSS)}, 2004, pp. 1--7.

\bibitem{mota2001dynamic}
W.~S. Mota and F.~L. Alvarado, ``Dynamic coupling between power markets and
  power systems with congestion constraints,'' in \emph{IEEE Porto PowerTech},
  2001, pp. 1--6.

\bibitem{kiani2010effect}
A.~Kiani and A.~Annaswamy, ``The effect of a smart meter on congestion and
  stability in a power market,'' in \emph{IEEE Conference on Decision and
  Control (CDC)}, 2010, pp. 194--199.

\bibitem{liang2015stability}
Y.~Liang, F.~Liu, and S.~Mei, ``Stability analysis of the hybrid dynamics
  coupling power systems with power markets,'' in \emph{IEEE Power \& Energy
  Society General Meeting (PESGM)}, 2015, pp. 1--5.

\bibitem{jokic2009real}
A.~Joki{\'c}, M.~Lazar, and P.~P. van~den Bosch, ``Real-time control of power
  systems using nodal prices,'' \emph{International Journal of Electrical Power
  \& Energy Systems}, vol.~31, no.~9, pp. 522--530, 2009.

\bibitem{stegink2017unifying}
T.~Stegink, C.~De~Persis, and A.~van~der Schaft, ``A unifying energy-based
  approach to stability of power grids with market dynamics,'' \emph{{IEEE
  Transactions on Automatic Control}}, vol.~62, no.~6, pp. 2612--2622, 2017.

\bibitem{stegink2017frequency}
T.~Stegink, A.~Cherukuri, C.~De~Persis, A.~van~der Schaft, and J.~Cort{\'e}s,
  ``Frequency-driven market mechanisms for optimal dispatch in power
  networks,'' \emph{arXiv preprint arXiv:1801.00137}, 2017.

\bibitem{stegink2018hybrid}
T.~Stegink, A.~Cherukuri, C.~De~Persis, A.~Van~der Schaft, and J.~Cort{\'e}s,
  ``Hybrid interconnection of iterative bidding and power network dynamics for
  frequency regulation and optimal dispatch,'' \emph{IEEE Transactions on
  Control of Network Systems}, vol.~6, no.~2, pp. 572--585, 2018.

\bibitem{kundur1994power}
P.~Kundur, N.~J. Balu, and M.~G. Lauby, \emph{Power System Stability and
  Control}.\hskip 1em plus 0.5em minus 0.4em\relax McGraw-hill New York, 1994,
  vol.~7.

\bibitem{shamma2005fictitious}
J.~S. {Shamma} and G.~{Arslan}, ``{Dynamic fictitious play, dynamic gradient
  play, and distributed convergence to Nash equilibria},'' \emph{IEEE
  Transactions on Automatic Control}, vol.~50, no.~3, pp. 312--327, 2005.

\bibitem{goldsztajn2019proximal}
D.~Goldsztajn, F.~Paganini, and A.~Ferragut, ``Proximal optimization for
  resource allocation in distributed computing systems with data locality,'' in
  \emph{IEEE Annual Allerton Conference on Communication, Control, and
  Computing (Allerton)}, 2019, pp. 773--780.

\bibitem{you2020saddle}
P.~You and E.~Mallada, ``Saddle flow dynamics: Observable certificates and
  separable regularization,'' \emph{arXiv preprint arXiv:2009.14714}, 2020.

\bibitem{dhingra2018proximal}
N.~K. Dhingra, S.~Z. Khong, and M.~R. Jovanovi{\'c}, ``{The proximal augmented
  Lagrangian method for nonsmooth composite optimization},'' \emph{IEEE
  Transactions on Automatic Control}, vol.~64, no.~7, pp. 2861--2868, 2018.

\bibitem{bacciotti2006nonpathological}
A.~Bacciotti and F.~Ceragioli, ``{Nonpathological Lyapunov functions and
  discontinuous Carath{\'e}odory systems},'' \emph{Automatica}, vol.~42, no.~3,
  pp. 453--458, 2006.

\end{thebibliography}

\begin{IEEEbiography}
	[{\includegraphics[width=1in,height=1.25in,clip,keepaspectratio]{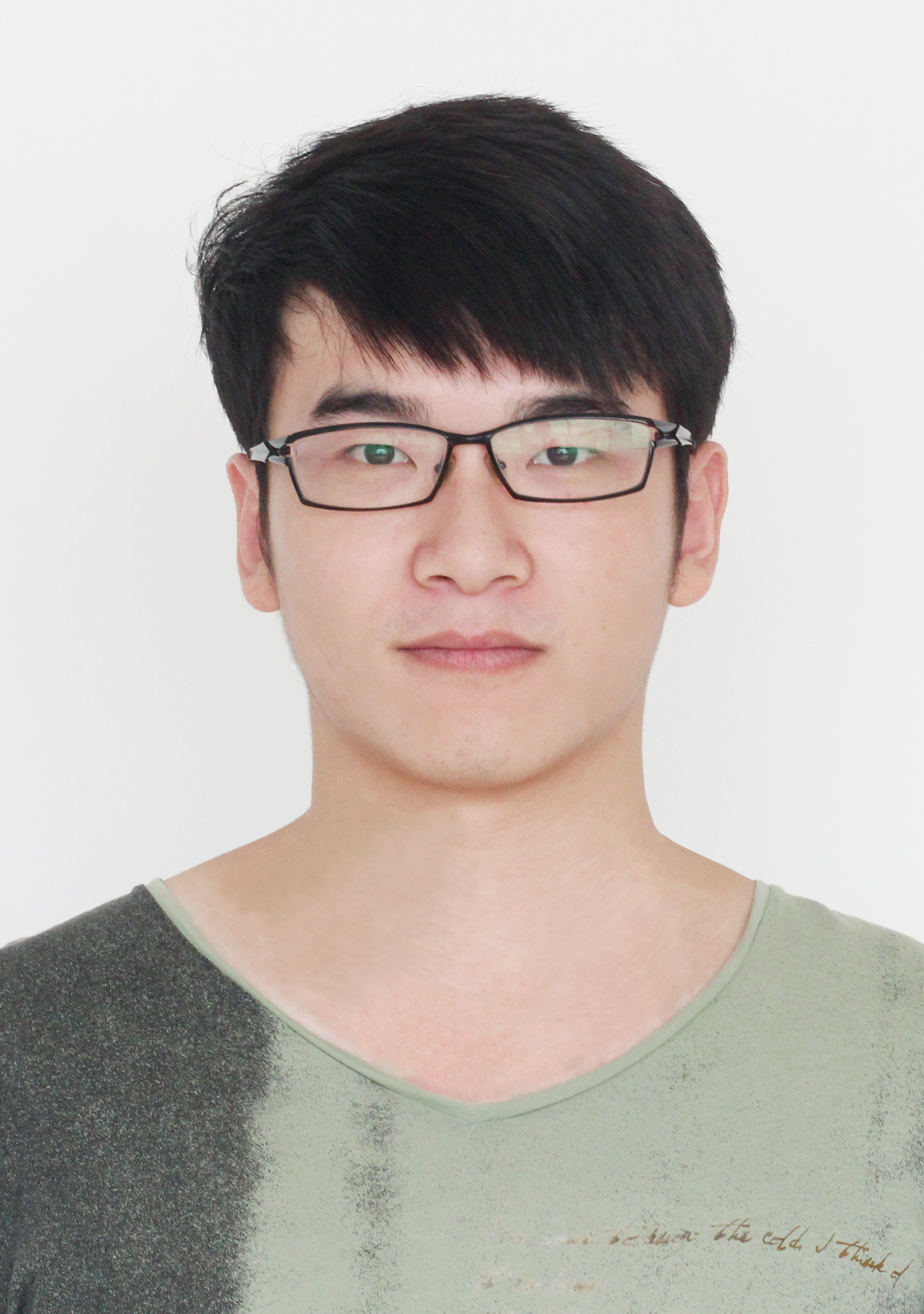}}]{Pengcheng You}
	(S'14-M’18) earned his B.S. in electrical engineering and Ph.D. in control, both from Zhejiang University, China, in 2013 and 2018, respectively. Currently he is a Postdoctoral Fellow with the Department of Electrical and Computer Engineering at Johns Hopkins University. His research focuses on power systems and electricity markets.
\end{IEEEbiography}

\begin{IEEEbiography}
	[{\includegraphics[width=1in,height=1.25in,clip,keepaspectratio]{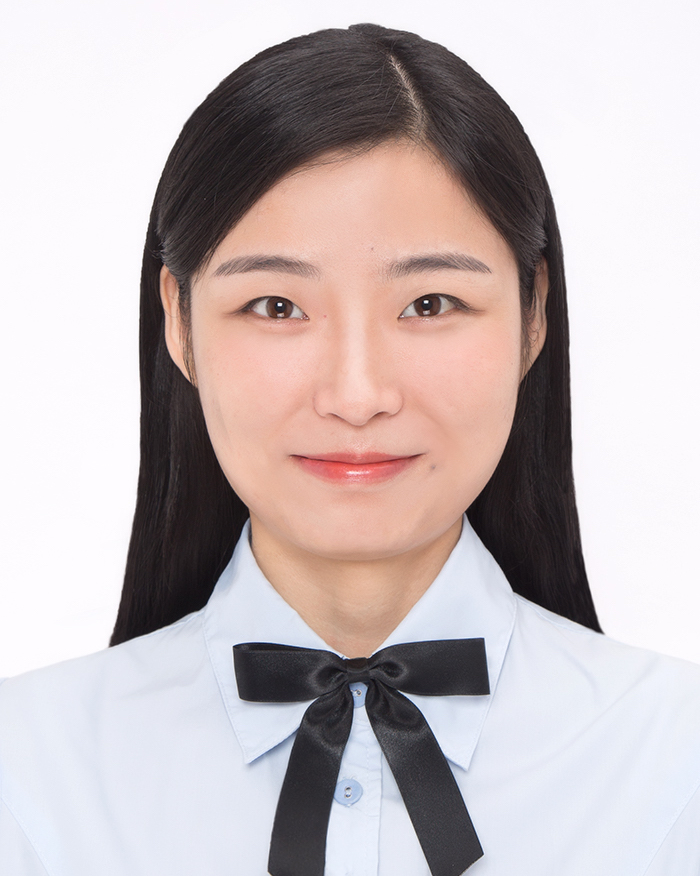}}]{Yan Jiang} received the Ph.D. degree at the Department of Electrical and Computer Engineering and the M.S.E. degree at the Department of Applied Mathematics and Statistics, Johns Hopkins University. She received the B.Eng. degree in electrical engineering and automation from Harbin Institute of Technology in 2013, and the M.S. degree in electrical engineering from Huazhong University of Science and Technology in 2016. Her research interests lie in the area of control of power systems.
\end{IEEEbiography}

\begin{IEEEbiography}
	[{\includegraphics[width=1in,height=1.25in,clip,keepaspectratio]{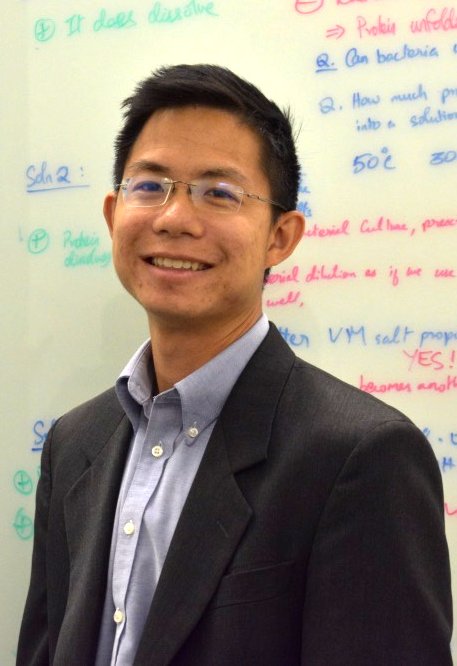}}]{Enoch Yeung} has a B.S. in Mathematics from Brigham Young University, \emph{magna cum laude} with university honors and a Ph.D. in Control and Dynamical Systems from the California Institute of Technology. He is an assistant professor in the Department of Mechanical Engineering at the University of California, Santa Barbara. He is a recipient of the Kanel Foundation Fellowship, the National Science Foundation Graduate Fellowship, the National Defense Science and Engineering Fellowship, the PNNL Project Team of the Year Award, the PNNL Outstanding Performance Award, the Keck Foundation Award, and an ARO Young Investigator Program Award. His research interests center on learning algorithms for dynamical systems, control theory, and control of biological networks.  
\end{IEEEbiography}

\begin{IEEEbiography}[{\includegraphics[width=1in,height=1.25in,clip,keepaspectratio]{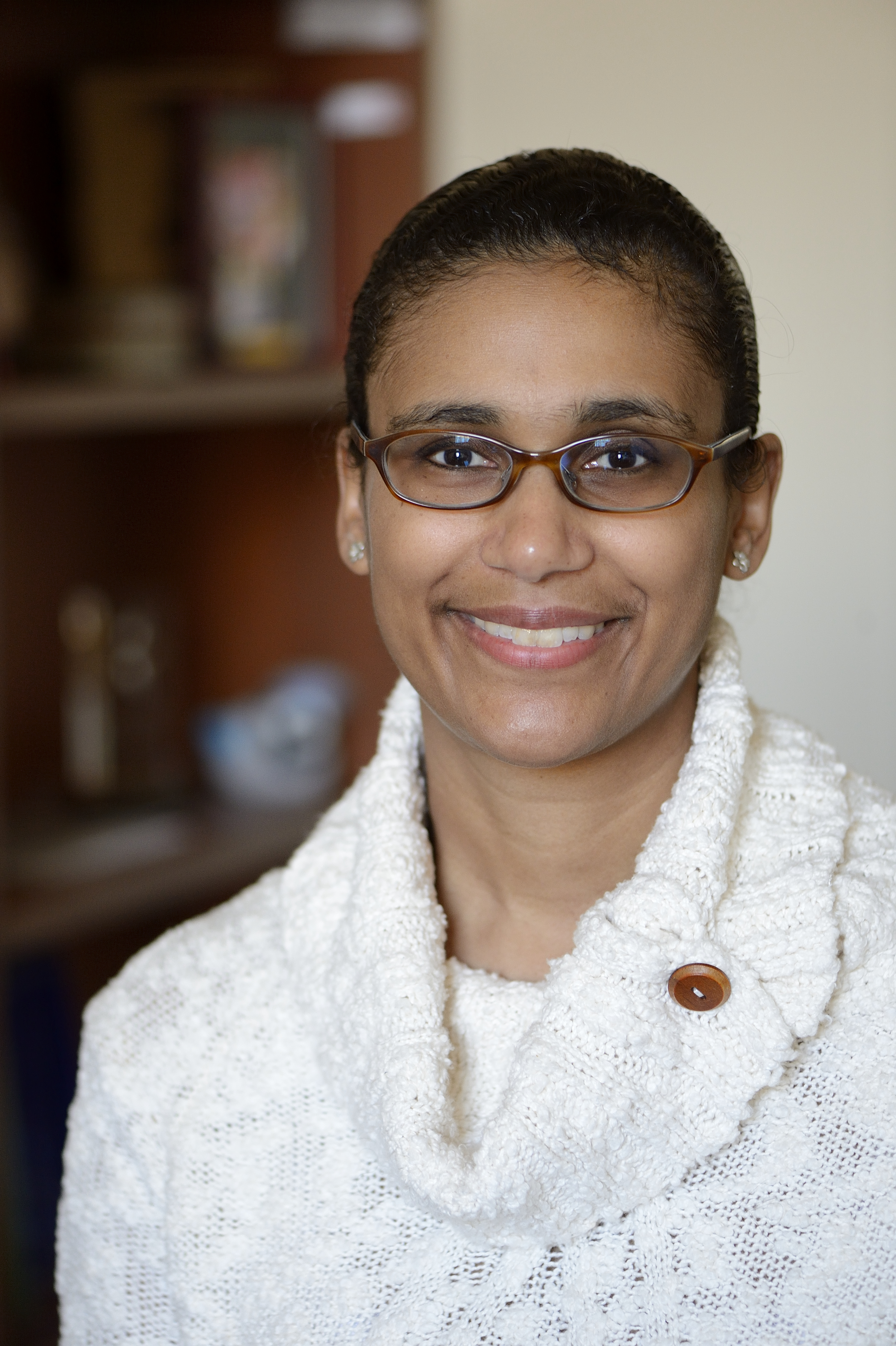}}]
{Dennice Gayme} (M'10-SM'14) an Associate Professor in Mechanical Engineering and the Carol Croft Linde Faculty Scholar at the Johns Hopkins University.  She earned her B. Eng. \& Society from McMaster University in 1997 and an M.S. from the University of California at Berkeley in 1998, both in Mechanical Engineering. She received her Ph.D. in Control and Dynamical Systems in 2010 from the California Institute of Technology.  She was a recipient of the JHU Catalyst Award in 2015, ONR Young Investigator and NSF CAREER awards in 2017, a JHU Discovery Award in 2019 and the 2020 Whiting School of Engineering Johns Hopkins Alumni Association Excellence in Teaching Award. Her research interests are in modeling, analysis and control for spatially distributed and large-scale networked systems in applications such as wall-bounded turbulent flows, wind farms, and power grids. 
\end{IEEEbiography}

\begin{IEEEbiography}[{\includegraphics[width=1in,height=1.25in,clip,keepaspectratio]{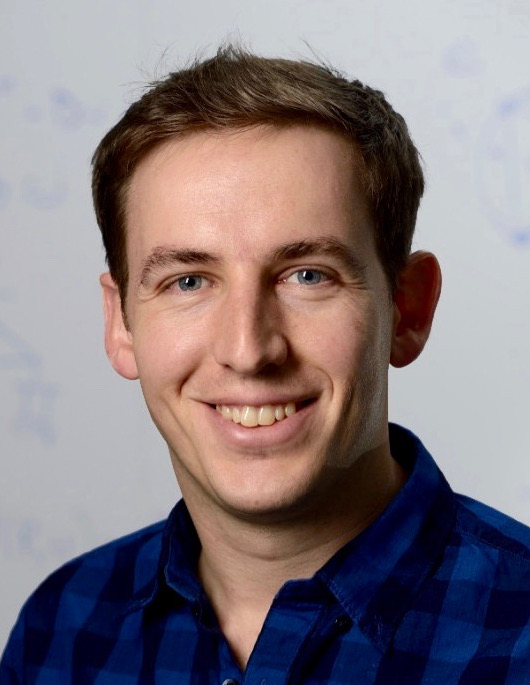}}]
{Enrique Mallada} (S'09-M'13-SM') is an Assistant Professor of Electrical and Computer Engineering at Johns Hopkins University. Prior to joining Hopkins in 2016, he was a Post-Doctoral Fellow in the Center for the Mathematics of Information at Caltech from 2014 to 2016. He received his Ingeniero en Telecomunicaciones degree from Universidad ORT, Uruguay, in 2005 and his Ph.D. degree in Electrical and Computer Engineering with a minor in Applied Mathematics from Cornell University in 2014. 
Dr. Mallada was awarded 
the NSF CAREER award in 2018,
the ECE Director's PhD Thesis Research Award for his dissertation in 2014, 
the Center for the Mathematics of Information (CMI) Fellowship from Caltech in 2014,
and the Cornell University Jacobs Fellowship in 2011. 
His research interests lie in the areas of control, dynamical systems and optimization, with applications to engineering networks such as power systems and the Internet.
\end{IEEEbiography}

\end{document}